\documentclass[10pt,fleqn]{article}
\usepackage[dvips]{epsfig}
\input{epsf}
\usepackage{amsmath,amssymb,amsfonts}
\usepackage{amscd}
\usepackage[dvips,matrix,arrow]{xy}

\usepackage[small,nohug,heads=vee]{diagrams}
\diagramstyle[labelstyle=\scriptstyle]
\usepackage{array,tabularx}

\numberwithin{equation}{subsection}
\numberwithin{figure}{subsection}

\newtheorem{theorem}{Theorem}[subsection]
\newtheorem{lemma}[theorem]{Lemma}
\newtheorem{remark}[theorem]{Remark}
\newtheorem{proposition}[theorem]{Proposition}
\newtheorem{corollary}[theorem]{Corollary}

\newtheorem{examples}[theorem]{Examples}

\newtheorem{remarks}[theorem]{Remarks}

\newenvironment{proof}[1][Proof:]{\begin{trivlist}
\item[\hskip \labelsep {\bfseries #1}]}{\end{trivlist}}

\newcommand{\qed}{\nobreak \ifvmode \relax \else
      \ifdim\lastskip<1.5em \hskip-\lastskip
      \hskip1.5em plus0em minus0.5em \fi \nobreak
      \vrule height0.75em width0.5em depth0.25em\fi}

 0 3 \font\ccc =msbm10 \font\cccc = msbm7

\date{ }

\begin{document}

\title{On the Preference Relations with \\
Negatively Transitive Asymmetric Part. I}

\author{Maria Viktorovna Droganova\footnote{Department of
Economics, Duke University, Durham, NC 27708-0097, USA}\\
Valentin Vankov Iliev\footnote{Institute of Mathematics and
Informatics, Bulgarian Academy of Sciences, Sofia 1113, Bulgaria}}

\maketitle

\begin{abstract}

Given a linearly ordered set $I$, every surjective map $p\colon A\to
I$ endows the set $A$ with a structure of set of preferences by
``replacing" the elements $\iota\in I$ with their inverse images
$p^{-1}(\iota)$ considered as ``balloons" (sets endowed with an
equivalence relation), lifting the linear order on $A$, and
``agglutinating" this structure with the balloons. Every ballooning
$A$ of a structure of linearly ordered set $I$ is a set of
preferences $A$ whose preference relation (not necessarily complete)
is negatively transitive and every such structure on a given set $A$
can be obtained by ballooning of certain structure of a linearly
ordered set $I$, intrinsically encoded in $A$. In other words, the
difference between linearity and negative transitivity is
constituted of balloons. As a consequence of this characterization,
under certain natural topological conditions on the set of
preferences $A$ furnished with its interval topology, the existence
of a continuous generalized utility function on $A$ is proved.

\end{abstract}

\section*{Introduction}

The aim of this paper is twofold: To characterize the preference
relations with negatively transitive asymmetric part and to prove
the existence of a continuous generalized utility function on the
corresponding set of preferences furnished with its interval
topology with respect to which it is connected and separable. Since
the authors can not find appropriate source(s) for citing, they
allow themselves to gather here from a unique point of view material
that can be found in various texts using varying terminology and
notation at different levels of rigor and generality. Moreover, the
systematic use of saturated subsets shortens and clarifies the
exposition. In what follows, we use, in general, the terminology
from~\cite{[35]}.

Below we use the terms ``preorder" as a synonym of ``preference
relation" and ``preordered set" as a synonym of ``set of
preferences".

Any preference relation $R$ on a set $A$ is a disjoint union of its
symmetric part $E=I_R$ (an equivalence relation) and its asymmetric
part $F=P_R$ (an asymmetric and transitive relation). Moreover, $F$
is $E$-saturated. Conversely, any ordered pair $(E,F)$ of such
relations produces a reflexive and transitive relation $R$ by
forming their union: $R=E\cup F$. The rule $R\mapsto (I_R,P_R)$
establishes a one-to-one correspondence between the set of of all
reflexive and transitive binary relations $R$ on $A$ and the set of
all ordered pairs $(E,F)$, where $E$ is an equivalence relation on
$A$, $F$ is an asymmetric transitive binary relation on $A$, which
is $E$-saturated, and $E\cap F=\emptyset$.

The nature of the connection between negative transitivity and
completeness is the main subject of the present paper. In
mathematical economics, the idea descends from Fishburn and his
Theorem 2.1 in~\cite[Ch. 2, Sec. 2.2]{[15]}. Under the assumptions
of completeness of the preference relation $R$ and negative
transitivity of its asymmetric part $F$, he notes that ${\cal E}=
E_F$ ($xE_Fy$ means ``$x$ and $y$ are not $F$-comparable) is an
equivalence relation and that the partition $\{F,{\cal E},F^{-1}\}$
of $A^2$ can be factorized with respect to ${\cal E}$, thus
producing a linearly ordered factor-set. In mathematics this idea is
older and one can find it, for example, in~\cite[Ch. III, Sec. 1,
Exercise 4]{[2]}, where N.~Bourbaki factorizes a partially ordered
set with respect to the transitive closure of the relation ``$x$ and
$y$ are not comparable or $x=y$" and obtains a linearly ordered
factor-set. The result is presented here, faintly generalized for a
preorder $R$, as Theorem~\ref{5.10.15}. If we strengthen the
condition of the later theorem by assuming that the asymmetric part
$F$ of $R$ is negative transitive, we obtain our Theorem~\ref{2.1.1}
which generalizes Fishburn's theorem. More precisely, the negative
transitivity allows us to show that $F$ is ${\cal E}$-saturated
(see~\ref{1.5}) and the triple $\{F,E,F^{-1}\}$, where $E$ is the
symmetric part of $R$, can be factored out with respect to the
equivalence relation ${\cal E}$. The triple of factors $\{F/{\cal
E},E/{\cal E},F^{-1}\!/{\cal E}\}$ is a partition of the factor-set
$A'=A/{\cal E}$, we have $E/{\cal E}=D_{A'}$ (the diagonal of
$A'^2$), $F^{-1}\!/{\cal E}=(F/{\cal E})^{-1}$, and $R'=D_{A'}\cup
(F/{\cal E})$ is a linear order on $A'$. Note that the ${\cal
E}$-equivalence classes are the fibres of the canonical surjective
map $c\colon A\to A'$, $x\mapsto\hbox{\rm\ the equivalence class of\
} x$. Since ${\cal E}=E\cup E_R$, every ${\cal E}$-equivalence class
is a disjoint union of $E$-equivalence classes, that is,
indifference curves, and the members of different indifference
curves in a fixed ${\cal E}$-equivalence class are not
$R$-comparable. In particular, every ${\cal E}$-equivalence class is
$E$-saturated and hence inherits from $E$ its own equivalence
relation. In case the later equivalence relation on every ${\cal
E}$-equivalence class is trivial, that is, every ${\cal
E}$-equivalence class coincides with an $E$-indifference curve, or,
what is the same, when ${\cal E}=E$, then $E_R=\emptyset$ and we
obtain the completeness of $R$ and Fishburn's Theorem 2.1.
Conversely, starting with a linearly ordered set $I$ and with a
family $(A_\iota)_{\iota\in I}$ of balloons, that is, sets $A_\iota$
each endowed with an equivalence relation $E_\iota$, $\iota\in I$,
we construct in Theorem~\ref{2.5.5} a set $A$ --- the coproduct of
the family $(A_\iota)_{\iota\in I}$ with its natural projection
$p\colon A\to I$ and a preference relation $R$ on $A$ with
negatively transitive asymmetric part $F=P_R$, such that the
equivalence relation $E=I_R$ induces on every balloon
$p^{-1}(\iota)=A_\iota$ the equivalence relation $E_\iota$.
Moreover, the equivalence relation associated with the partition
$(A_\iota)_{\iota\in I}$ of $A$ is identical to ${\cal E}=E_F$, the
linear order on $I$ is the factor-relation of $R$, and the
corresponding strict order on $I$ is the factor-relation of $F$ with
respect to ${\cal E}$. Thus, every ballooning $A$ of a structure of
linearly ordered set $I$ has a structure of set of preferences $A$
whose preference relation has negatively transitive asymmetric part
and every such structure on a given set $A$ can be obtained by
ballooning of certain structure of a linearly ordered set $I$,
intrinsically encoded in $A$.

G.~Debreu (not without influence from N.~Bourbaki) studies
in~\cite{[10]} the existence of a continuous utility representation
$u\colon A\to \hbox{\ccc R}$ of a complete set of preferences $A$.
Lemma~\ref{2.10.1} reduces the problem to the existence of a
strictly increasing and continuous map of the factor-set $A/E$ into
the real line $\hbox{\ccc R}$. Under the condition of connectedness
and separability on $A$ endowed with a ``natural" topology (that is,
topology which is finer than the interval topology on $A$),
G.~Debreu proves this existence by refereing to~\cite[(6.1)]{[10]}.

In his monograph~\cite{[15]}, P.~C. Fishburn proves (Theorem 2.2)
that if the set $A$ is furnished with an asymmetric relation $F$
which is negatively transitive and if the factor-set $A/{\cal E}$ is
countable, then there exists a utility function $u\colon
A\to\hbox{\ccc R}$. Further, in~\cite[Theorem 3.1]{[15]} he presents
a necessary and sufficient condition for existence of such utility
function, which is, in fact, the Birkhoff's criterium~\cite[Ch. III,
Theorem 2]{[1]} for the linearly ordered set $A/{\cal E}$ to be
embedded in a order-preserving way into the real line.

In~\cite[Sec. 5]{[32]}, E.~A. Ok argues that ``...utility theory can
be beneficially extended to cover incomplete preference relations"
and, moreover, ``Identifying a useful set of conditions on preorders
that would lead to such a representation result [that is, functional
representation of preorders] is an open problem worthy of
investigation".

Our setup uses the sufficient condition (connectedness and
separability) from~\cite[Ch. IV, Sec. 2, Exercise 11 a]{[5]} for the
the existence of a homeomorphic embedding of factor-set $A/{\cal E}$
into the real line (see also Theorem~\ref{1.8.80}). If the
preference relation on $A$ has negatively transitive asymmetric
part, there is a close bond between the topologies of $A$ and
$A/{\cal E}$ described in Proposition~\ref{5.20.1}. Using this
strong relationship, in case $A$ is connected and separable we prove
the existence of a continuous generalized utility function on $A$
(Theorem~\ref{2.10.5}).

The aim of the appendix is to give a proof of Theorem~\ref{1.8.80}
which asserts that any linearly ordered set $A$ endowed with the
interval topology, which is connected and separable possesses a
strictly increasing homeomorphism onto one of the intervals in the
real line with endpoints $0$ and $1$. The text represents the
slightly dressed up notes of the authors during reading of~\cite[Ch.
III, Sec. 1]{[2]}, the problems from~\cite[Ch. IV, Sec. 2, Exercises
6, 7, 8, 9, 11]{[5]}, and other relevant literature.

Portions of this paper are parts of the senior thesis of the first
author made under the supervision of the second one. We intend to
devote Part 2 of the paper to some economic applications of our
results.

\section{Mathematical Background}

\label{1}

Below, for the sake of completeness and non-ambiguous reading of
this paper, we present the main statements which are used. Moreover,
we show how saturatedness appears naturally in the context of
factorization.

If the opposite is not stated explicitly, we assume that the
occurrence of the variables $x$, $y$, $z,\ldots$ in a statement is
bounded by the universal quantifier (``for all").

\subsection{Factorizations}

\label{1.5}

Let $A$ be a non-empty set and let $E$ be an equivalence relation on
$A$. For $x\in A$ we denote by $C_x$ the equivalence class with
respect to $E$ (or, $E$-equivalence class) with representative $x$.
The subset $B\subset A$ is said to be \emph{saturated with respect
to $E$}, or \emph{$E$-saturated}, if $x\in B$ implies $C_x\subset
B$. The factor-set $A/\!E$ (that is, the set of all equivalent
classes) is a partition of $A$ and $c\colon A\to A/\!E$, $x\mapsto
C_x$, is the canonical surjective map. Sometimes we write
$\bar{x}=c(x)$.

The relation $R$ is called \emph{left $E$-saturated} if $xRy$ and
$x'Ex$ imply $x'Ry$, \emph{right $E$-saturated} if $xRy$ and $y'Ey$
imply $xRy'$, and \emph{$E$-saturated} if $xRy$, $x'Ex$, and $y'Ey$
imply $x'Ry'$.

Let $E$ be an equivalence relation on $A$. The product ${\cal
E}=E\times E$ is an equivalent relation on $A\times A$ and let
$c_2\colon A\times A\to (A\times A)/\!{\cal E}$ be the corresponding
canonical surjective map. The bijection $(c(x),c(y))\mapsto
c_2(x,y)$ identifies the sets $(A/\!E)\times (A/\!E)$ and $(A\times
A)/\!{\cal E}$ and then the canonical surjective map
\[
c_2\colon A\times A\to (A/\!E)\times (A/\!E),\hbox{\
}c_2(x,y)=(\bar{x},\bar{y}),
\]
induces via the rule $X\mapsto c_2(X)$, $X\subset A\times A$, a
canonical surjective map
\begin{equation}
c_2\colon 2^{A\times A}\to 2^{\left(A/\!E\right)\times
\left(A/\!E\right)}.\label{1.5.1}
\end{equation}

The map $c_2$ from~(\ref{1.5.1}) establishes a bijection
$R\mapsto\bar{R}$, where $\bar{R}=c_2(R)$, between the set of all
$E$-saturated binary relations $R$ on $A$ and the set of all binary
relations $\bar{R}\subset (A/\!E)\times (A/\!E)$ on the factor-set
$A/\!E$. We have $xRy$ if and only if $\bar{x}\bar{R}\bar{y}$.

Thus, any $E$-saturated binary relation $R$ on the set $A$ produces
a relation $\bar{R}=c_2(R)$ on the factor-set $A/E$ which we call
\emph{factor-relation} of $R$ with respect to the equivalence
relation $E$. We have $xRy$ if and only if $\bar{x}\bar{R}\bar{y}$.

Below we present another natural factorization $R'$ of a binary
relation $R$ on a set $A$ with respect to an equivalence relation
$E$. Let $A'=A/\!E$ be the factor-set and let $A\to A'$, $x\mapsto
\bar{x}$, be the canonical surjective map. We define the binary
relation $R'$ on $A'$ by the rule
\[
\bar{x}R'\bar{y}\hbox{\rm\ if for any\ } x'\in \bar{x} \hbox{\rm\
there exists\ }y'\in \bar{y} \hbox{\rm\ such that\ } x'Ry'.
\]
and call $R'$ \emph{weak factor-relation} of $R$ with respect to the
equivalence relation $E$, thus generalizing the notion of
factor-relation introduced above. This is also a generalization
of~\cite[Ch. III, Sec. 1, Exercise 2, a]{[2]}.

\subsection{Derivative relations}

\label{1.1}

Let $R$ be a binary relation on a non-empty set $A$. We denote by
$R^{-1}$ the inverse relation and by $R^c$ the complementary
relation. We set $I_R=R\cap R^{-1}$, $U_R=R\cup R^{-1}$,
$P_R=R\backslash I_R$ , and $E_R=U_R^c$. The relation $I_R$ is the
symmetric part of $R$ and the relation $P_R$ is the asymmetric part
of $R$. In case $R$ is reflexive and transitive (that is, a
preorder) $I_R$ is an equivalence relation and $P_R$ is asymmetric
and transitive. When, in addition, $I_R=D_A$, where $D_A$ is the
diagonal of $A^2$, the preorder $R$ is a partial order. A
\emph{balloon} is a preordered set $A$ furnished with symmetric
preorder $R$. Thus, $R$ is an equivalence relation: $R=I_R$.

\subsection{Saturatedness and Transitivity}

\label{5.1}

\begin{lemma} \label{5.1.1} Let $A$ be a set, let $R$ be
a reflexive binary relation on $A$, and let $E$ be an equivalence
relation on $A$. If $R\subset E$, then $R$ is weakly $E$-saturated.

\end{lemma}

\begin{proof} Let  $x'Ex$ and $xRy$. Then $x'Ey$ and if we set
$y'=x'$, then $yEy'$ and $x'Ry'$ because $R$ is reflexive.

\end{proof}

\begin{proposition} \label{5.1.5} Let $R$ be a binary relation on
the set $A$ and let $E$ be an equivalence relation on $A$. Let
$A'=A/E$ be the factor-set and let $c\colon A\to A'$, $x\mapsto
\bar{x}$, be the canonical surjective map. Then the canonical map
$c$ is increasing if and only if $R$ is weakly $E$-saturated.

Under the condition that $R$ is weakly $E$-saturated, the following
statements hold:

{\rm (i)} If $R$ is reflexive, then $R'$ is reflexive.

{\rm (ii)} If $R$ is transitive, then $R'$ is transitive.

{\rm (iii)} The relation $R'$ is antisymmetric if and only if $R$
satisfies the condition
\begin{equation}
xRy,\hbox{\ } yRz,\hbox{\rm\ and\ } xEz \hbox{\rm\ imply\ } xEy.
\label{5.1.10}
\end{equation}

{\rm (iv)} If $R$ is $E$-saturated, then $R'$ coincides with the
factor-relation $\bar{R}$. If, in addition, $I_R\subset E$, then
condition~(\ref{5.1.10}) holds.

{\rm (v)} Let $B$ be a set endowed with a binary relation $S$. For
any increasing map $p\colon A\to B$ which is constant on the members
of the factor-set $A'=A/E$ there exists a unique increasing map
$p'\colon A'\to B$ such that the diagram
\begin{diagram}
A           &         & \\
\dTo_{c} & \rdTo^{p}& \\
A'        & \rTo^{p'} & B\\
\end{diagram}
is commutative. If $p$ is surjective, then $p'$ is surjective. If,
in addition,
\begin{equation}
E=\{(x,y)\in A^2\mid p(x)=p(y)\},\label{5.1.15}
\end{equation}
then $p'$ is an increasing bijection.

\end{proposition}

\begin{proof} Let us suppose that $c$ is an increasing map, let $xRy$,
and let $x'E x$, that is, $x'\in\bar{x}$. We have $\bar{x}R'\bar{y}$
which means there exists $y'\in\bar{y}$ with $x'Ry'$ and hence $R$
is weakly $E$-saturated. Conversely, if $R$ is weakly $E$-saturated
and if $xRy$, then for any $x'\in\bar{x}$ there exists
$y'\in\bar{y}$ with $x'Ry'$, therefore $\bar{x}R'\bar{y}$.

The proofs of parts {\rm (i)} and {\rm (ii)} are immediate.

{\rm (iii)} Let $R$ satisfies condition~(\ref{5.1.10}) and let
$\bar{x}R'\bar{y}$, $\bar{y}R'\bar{x}$. In other words, there exist
$y'\in \bar{y}$ and $x'\in \bar{x}$, such that $xRy'$ and $y'Rx'$.
Since $xEx'$, we obtain $xEy'$, that is, $\bar{x}=\bar{y}$.
Conversely, let $R'$ be antisymmetric and let us suppose that $xRy$,
$yRz$, and $xEz$. In particular, $\bar{x}R'\bar{y}$,
$\bar{y}R'\bar{x}$, and this yields $\bar{x}=\bar{y}$, that is,
$xEy$.

{\rm (iv)} We have $\bar{R}\subset R'$. Now, let $\bar{x}R'\bar{y}$
and let $x'\in\bar{x}$, $y'\in\bar{y}$. We have that there exists
$y''\in\bar{y}$ with $x'Ry''$. Since $y''Ey'$, the saturatedness of
$R$ yields $x'Ry'$. Thus, we have $\bar{x}\bar{R}\bar{y}$.

Now, let us suppose, in addition, that $I_R\subset E$ and let $xRy$,
$yRz$, and $xEz$. Since $R$ is $E$-saturated, we obtain $yRx$, hence
$xI_Ry$ and this, in turn, implies $xEy$.

{\rm (v)} For any map $p\colon A\to B$ which is constant on the
members of the factor-set $A'=A/E$ there exists a unique map
$p'\colon A'\to B$ such that $p=p'\circ c$. Let $\bar{x}R'\bar{y}$
and we can suppose that $xRy$. Then $p(x)Sp(y)$ and this relation
can be rewritten as $p'(\bar{x})Sp'(\bar{y})$. Therefore $p'$ is an
increasing map. Since $p(x)=p'(\bar{x})$ the surjectivity of $p$
implies the surjectivity of $p'$. Under the
condition~(\ref{5.1.15}), if $p(x)=p(y)$, then $\bar{x}=\bar{y}$,
that is, $p'$ is a bijection.

\end{proof}

\begin{proposition}\label{5.1.65} Let $A$ be a set and let $F$ be
an asymmetric and negatively transitive relation on $A$. Then ${\cal
E}=E_F$ is an equivalence relation and the relation $F$ is
transitive and ${\cal E}$-saturated.

\end{proposition}

\begin{proof} The proofs that ${\cal E}$ is an equivalence relation
and that the asymmetry and negative transitivity of $F$ imply
transitivity of $F$ are straightforward. Since $\{{\cal
E},F,F^{-1}\}$ is a partition of $A^2$, we obtain that $F$ is ${\cal
E}$-saturated. Indeed, let $x'{\cal E}x$ and $xFy$. If $x'{\cal E}y$
(respectively, $yFx'$), then $x{\cal E}y$ (respectively, $xFx'$) ---
an absurdity. Thus, $x'Fy$. Similarly, let $xFy$ and $y{\cal E}y'$.
If $x{\cal E}y'$ (respectively, $y'Fx$), then $x{\cal E}y$
(respectively, $y'Fy$) --- an absurdity. Thus, $xFy'$.

\end{proof}

\subsection{Indifference and Completeness}

\label{5.10}

Let $A$ be a set. A binary relation $R$ on $A$ is said to be
\emph{indifference} if $R$ is reflexive and symmetric. In case $R$
is an indifference on $A$, its transitive closure
$R^{\left(t\right)}$ is an equivalence relation on $A$ and the
$R^{\left(t\right)}$-equivalence classes are called
\emph{$R$-indifference curves}. Thus, the factor-set
$A/\!R^{\left(t\right)}$ consists of all $R$-indifference curves. In
case the relation $R$ is self-understood, we denote by $\bar{x}$ the
$R$-indifference curve with representative $x\in A$.

\begin{lemma}\label{5.10.5} Let the binary relation $R$ on $A$
be asymmetric and transitive.

{\rm (i)} The relation $E_R$ is an indifference.

{\rm (ii)} If  $\bar{x}$ and $\bar{y}$ are two different
$R$-indifference curves, then any two elements $x'\in\bar{x}$ and
$y'\in\bar{y}$ are $R$-comparable and if $xRy$, then $x'Ry'$.

\end{lemma}

\begin{proof} {\rm (i)} The proof is immediate.

{\rm (ii)} Since $\bar{x}\neq\bar{y}$, any two elements
$x'\in\bar{x}$ and $y'\in\bar{y}$ are $R$-comparable. Let, for
example, $xRy$ and let $y'\in\bar{y}$. Let $y=b_0,b_1,\ldots,
b_n=y'$ be the finite sequence such that $b_iE_Rb_{i+1}$,
$i=0,1,\ldots,n-1$. Note that $b_i\in \bar{y}$ for all
$i=0,1,\ldots,n-1$. We use induction with respect to $i$ to prove
that $xRb_i$. This statement is true for $i=0$ and let us suppose
that $xRb_{i-1}$. If $b_iRx$, then the transitivity of $R$ imply
$b_iRb_{i-1}$ --- a contradiction. Thus, $xRb_i$ for any
$i=0,1,\ldots,n$ and, in particular, $xRy'$ for $y'\in\bar{y}$.
Using the above argumentation for $R^{-1}$ (which is also reflexive
and transitive) and taking into account that $E_R=E_{R^{-1}}$, we
obtain that $xRy$ and $x'\in\bar{x}$ yield $x'Ry$. Thus, if $xRy$,
$x'\in\bar{x}$, and $y'\in\bar{y}$, then $x'Ry'$.

\end{proof}

\begin{theorem}\label{5.10.15} Let the binary relation $R$ on $A$
be reflexive and transitive, let $F=P_R$ be its asymmetric part, and
let $E_F^{\left(t\right)}$ be the transitive closure of the
indifference $E_F$. One has:

{\rm (i)} The relation $R$ is weakly
$E_F^{\left(t\right)}$-saturated.

{\rm (ii)} The week factor-relation $R'$ on the factor-set
$A'=A/\!E_F^{\left(t\right)}$ is reflexive, transitive,
antisymmetric, and complete; $A'$ endowed with $R'$ is a linearly
ordered set.

\end{theorem}

\begin{proof} {\rm (i)} Let $xRy$ and $xE_F^{\left(t\right)}x'$.

Case 1. $\bar{x}=\bar{y}$.

In this case $yE_F^{\left(t\right)}x'$ and it is enough to note that
$x'Rx'$.

Case 2. $\bar{x}\neq\bar{y}$.

Lemma~\ref{5.10.5}, {\rm (ii)}, shows that $x$ and $y$ are
$R$-comparable and if $xRy$, then, in particular, for any
$x'\in\bar{x}$ there exists $y'\in\bar{y}$ such that $x'Ry'$.

{\rm (ii)} The weak factor-relation $R'$ on the factor-set $A'$ is
reflexive and transitive because of Proposition~\ref{5.1.5}, {\rm
(i)}, {\rm (ii)}. In accord with Lemma~\ref{5.10.5}, {\rm (ii)}, if
$\bar{x}\neq\bar{y}$, then $x$ and $y$ are $R$-comparable, say
$xRy$, and we obtain $\bar{x}R'\bar{y}$. Therefore the week
factor-relation $R'$ on $A'$ is complete. Let $\bar{x}R'\bar{y}$ and
$\bar{y}R'\bar{x}$. If we suppose $\bar{x}\neq\bar{y}$, then
Lemma~\ref{5.10.5}, {\rm (ii)}, yields $xRy$ and $yRx$ which
contradicts the asymmetry of $R$. Thus, $\bar{x}=\bar{y}$ and $R'$
is an antisymmetric relation. In other words, $R'$ is a linear order
on $A'$.

\end{proof}

\subsection{Coproducts of preordered sets}

\label{5.15}

Here we remind the definition of coproduct of a family
$(A_\iota)_{\iota\in I}$ of sets (see, for example,~\cite[Ch. II,
Sec. 4, $n^o\ 8$]{[2]}) and endow this coproduct with a structure of
preordered set, given such structures on the members $A_\iota$ of
the family and given a structure of partially ordered set on the
index set $I$.

For any $\iota\in I$ there exists a bijection $A_\iota\to
A_\iota\times\{\iota\}$, $x\mapsto (x,\iota)$. The set
$A=\cup_{\iota\in I}(A_\iota\times \{\iota\})$ is called the
\emph{coproduct} of the family $(A_\iota)_{\iota\in I}$.  Let
$U=\cup_{\iota\in I}A_\iota$ be the union of this family. In case
$A_\iota\cap A_\kappa=\emptyset$ for any $\iota\neq\kappa$ there
exists a projection $p\colon U\to I$, $x\mapsto\iota$, where $x\in
A_\iota$. We have a bijection between the union $U$ and the
coproduct $A$, given by the rule
\[
U\to A,\hbox{\ } x\mapsto (x,p(x)).
\]
When the family $(A_\iota)_{\iota\in I}$ has pairwise disjoint
members we identify the union $U$ and the coproduct $A$ via the
above bijection and define the \emph{projection} $p\colon A\to I$,
$x\mapsto\iota$ for $x\in A_\iota$.

Below, when discussing the coproduct of a family of sets, we
implicitly suppose without any loss of generality that the members
of this family are pairwise disjoint. In other words, ``coproduct"
is a shorthand for ``union of pairwise disjoint sets".

\begin{proposition}\label{5.15.1} Let $(A_\iota)_{\iota\in I}$ be a
family of preordered sets, let $(R_\iota)_{\iota\in I}$ be the
family of the corresponding preorders, and let the index set $I$ be
partially ordered. Let $A=\coprod_{\iota\in I}A_\iota$ be the
coproduct of the family $(A_\iota)_{\iota\in I}$, let $p\colon A\to
I$, $x\mapsto\iota$ for $x\in A_\iota$, be the natural projection,
and let ${\cal E}$ be the equivalence relation associated with the
partition $(A_\iota)_{\iota\in I}$ of the coproduct $A$.

{\rm (i)} The coproduct $A$ of this family has a structure of
preordered set with preorder $R$ defined by the rule
\[
xRy \hbox{\rm\ if\ } "p(x) < p(y)" \hbox{\rm\ or\ } "p(x)=p(y)
\hbox{\rm\ and\ } xR_{p\left(x\right)}y",
\]
where $<$ is the asymmetric part of the partial order $\leq$ on $I$.

{\rm (ii)} One has
\[
xI_Ry \hbox{\rm\ if and only if\ }"p(x)=p(y)" \hbox{\rm\ and\ }
"xI_{R_{p\left(x\right)}}y".
\]

{\rm (iii)} One has
\[
xP_Ry \hbox{\rm\ if and only if\ } "p(x)<p(y)" \hbox{\rm\ or\ }
"p(x)=p(y) \hbox{\rm\ and\ } xP_{R_{p\left(x\right)}}y".
\]

Let, in addition, $(A_\iota)_{\iota\in I}$ be a family of balloons.
Then one has:

{\rm (iv)} $xP_Ry$ if and only if $p(x)<p(y)$.

{\rm (v)} The asymmetric part $P_R$ of the preorder $R$ is ${\cal
E}$-saturated.

\end{proposition}

\begin{proof}
{\rm (i)} We have $xRx$ because $xR_\iota x$ for $\iota=p(x)$. Now,
let $xRy$ and $yRz$. We obtain $p(x)\leq p(y)$ and $p(y)\leq p(z)$,
hence $p(x)\leq p(z)$. If $p(x)<p(z)$, then $xRz$. Otherwise, there
exists $\iota\in I$ such that $\iota=p(x)=p(y)=p(z)$, $xR_\iota y$,
and $yR_\iota z$. Therefore $xR_\iota z$ and this yields $xRz$.

{\rm (ii)} We have $xRy$ and $yRx$ if and only if there exists
$\iota\in I$ such that $p(x)=p(y)=\iota$, $xR_\iota y$, and
$yR_\iota x$.

{\rm (iii)} This is an immediate consequence of parts {\rm (i)} and
{\rm (ii)}.

{\rm (iv)} We have $P_{R_{p\left(x\right)}}=\emptyset$ and part {\rm
(iii)} yields that $xP_Ry$ is equivalent to $p(x) < p(y)$.

{\rm (v)} If $x'{\cal E}x$ and $y{\cal E}y'$ (that is, $x'\in
A_{p\left(x\right)}$ and $y'\in A_{p\left(y\right)}$) and if
$xP_Ry$, then $x'P_Ry'$ because of part {\rm (iv)}.

\end{proof}

Let $(A_\iota)_{\iota\in I}$ be a family  of preordered sets, whose
index set $I$ is partially ordered. The coproduct $A$ of this
family, endowed with the preorder from Proposition~\ref{5.15.1},
{\rm (i)}, is said to be the \emph{coproduct of the preordered sets
$(A_\iota)$, ${\iota\in I}$}, or the \emph{coproduct of the family
$(A_\iota)_{\iota\in I}$ of preordered sets}.

We note that definition from part {\rm (i)} of the above proposition
is a generalization of the definition from~\cite[Ch. 3, Sec. 1,
Exercise 3, a]{[2]}.

\subsection{Interval topology of coproduct of a family of balloons}

\label{5.20}

Here we use freely the terminology, notation, and results from
subsection~\ref{5.15} as well as from~\cite[Ch. I]{[5]}).

Let $A$ be a preordered set and let $R$ be its preorder. Let us set
$E=I_R$ and $F=P_R$. An \emph{open interval in $A$} is any subset of
$A$ of the form $(x,y)=\{z\in A\mid xFz\hbox{\rm\ and\ }zFy\}$, or
$(\leftarrow, x)=\{z\in A\mid zFx\}$, or $(x,\rightarrow)=\{z\in
A\mid xFz\}$, where $x,y\in A$. A \emph{closed interval in $A$} is
any subset of $A$ of the form $[x,y]=\{z\in A\mid xRz\hbox{\rm\ and\
}zRy\}$, or $(\leftarrow, x]=\{z\in A\mid zRx\}$, or
$[x,\rightarrow)=\{z\in A\mid xRz\}$, where $x,y\in A$. In order to
emphasize that the corresponding interval is in $A$, we write
$(x,y)=(x,y)_A$, etc. The \emph{interval topology} ${\cal T}_0(A)$
on $A$ is the topology generated by the set ${\cal O}_A$ of all open
intervals in $A$, that is, $U\in {\cal T}_0(A)$ ($U$ is \emph{open})
if $U$ is a union of finite intersections of open intervals in $A$.
In particular, the open (respectively, closed) intervals are open
(respectively, closed) sets in the interval topology. When $A$ is a
linearly ordered set, any finite intersection of open intervals is
an open interval, hence in this case the set ${\cal O}_A$ is a base
of the topology of $A$.

We remind that a topological space $A$ is said to be
\emph{separable} if there is a countable subset $D\subset A$ such
that $D\cap U\neq\emptyset$ for any non-empty open subset $U\subset
A$. In order to prove separability, it is enough to verify the
inequality $D\cap U\neq\emptyset$ for any non-empty member $U$ of a
base of the topology of $A$.

\begin{proposition} \label{5.20.1} Under the conditions of
Proposition~\ref{5.15.1}, let the index set $I$ be linearly ordered,
let $(A_\iota)_{\iota\in I}$ be a family of balloons, and let $A$
and $I$ be endowed with their interval topologies.

{\rm (i)} The natural projection $p\colon A\to I$ is a strictly
increasing map which induces a pair of mutually inverse bijections
\begin{equation}
{\cal O}_A\to {\cal O}_I,\hbox{\ } J\mapsto p(J), \label{5.20.5}
\end{equation}
and
\begin{equation}
{\cal O}_I\to {\cal O}_A,\hbox{\ }  K\mapsto p^{-1}(K).
\label{5.20.10}
\end{equation}

{\rm (ii)} The set ${\cal O}_A$ is a base of the topology of $A$.

{\rm (iii)} $p\colon A\to I$ is a continuous and open map.

\end{proposition}

\begin{proof} Let $F=P_R$ be the asymmetric part of the preorder $R$
on the coproduct $A$.

{\rm (i)} According to Proposition~\ref{5.15.1}, {\rm (iv)}, {\rm
(v)}, the relations $xFy$ and $p(x)<p(y)$ are equivalent and hence
any open interval $J$ in $A$ is ${\cal E}$-saturated. In particular,
$p$ is a strictly increasing map. Moreover, $p$ maps any open
interval $J$ in $A$ into an open interval $K$ in $I$ whose
endpoint(s) is (are) the image(s) of the endpoint(s) of $J$. Since
the sets $A_\iota$, $\iota\in I$, are not empty, the natural
projection $p$ is a surjective map and we have $p(J)=K$. On the
other hand, the inverse image $p^{-1}(K)$ of an open interval $K$ in
$I$ is an open interval in $A$ with endpoint(s) chosen to be inverse
image(s) of the endpoint(s) of $K$. Thus, the existence of the
maps~(\ref{5.20.5}) and~(\ref{5.20.10}) is justified. Finally, we
have $p(p^{-1}(K))=K$ and the ${\cal E}$-saturatedness of the open
intervals $J$ in $A$ yields $J=p^{-1}(p(J))$.

{\rm (ii)} It is enough to show that a finite intersection of open
intervals in $A$ is an open interval. For let $J_\lambda\in {\cal
O}_A$, $\lambda=1,\ldots,s$ and let $K_\lambda\in {\cal O}_I$ be
such that $J_\lambda=p^{-1}(K_\lambda)$. Since $I$ is a linearly
ordered set, the intersection $K=\cap_{\lambda=1}^sK_\lambda$ is an
open interval in $I$ and we have
\[
p^{-1}(K)=p^{-1}(\cap_{\lambda=1}^sK_\lambda)=
\cap_{\lambda=1}^sp^{-1}(K_\lambda)=\cap_{\lambda=1}^sJ_\lambda.
\]
Now, part~{\rm (i)} yields the statement.

{\rm (iii)} Apply part~{\rm (i)}.

\end{proof}

\begin{corollary} \label{5.20.15} {\rm (i)} The topology of $A$ is
the inverse image via $p$ of the topology of $I$.

{\rm (ii)} The topological space $A$ is connected if and only if the
topological space $I$ is connected.

{\rm (iii)} The topological space $A$ is separable if and only if
the topological space $I$ is separable and the countable dense
subsets $D\subset A$, $D'\subset I$ can be chosen such that
$p(D)=D'$.

\end{corollary}

\begin{proof} {\rm (i)} Proposition~\ref{5.20.1} implies
immediately that ${\cal T}_0(A)=\{p^{-1}(V)\mid V\in {\cal T}_0(I)$.

{\rm (ii)} According to Proposition~\ref{5.20.1}, {\rm (iii)}, the
surjective map $p\colon A\to I$ is continuous, hence the
connectedness of $A$ implies the connectedness of $I$ (see~\cite[Ch.
I, Sec. 11, n$^o$, Proposition 4]{[5]}. Now, let $I$ be a connected
topological space and let us suppose that $A$ is not connected. In
other words, there exist two open non-empty and disjoint subsets $U$
and $U'$ of $A$, such that $A=U\cup U'$. Proposition~\ref{5.20.1},
{\rm (ii)}, yields that $U=\cup_{\lambda\in L}J_\lambda$ and
$U'=\cup_{\mu\in M}J'_\mu$ for some non-empty families
$(J_\lambda)_{\lambda\in L}$, $(J'_\mu)_{\mu\in M}$ of open
intervals in $A$. Let $(K_\lambda)_{\lambda\in L}$,
$(K'_\mu)_{\mu\in M}$ be the families of open intervals in $I$, such
that $J_\lambda=p^{-1}(K_\lambda)$ and $J'_\mu=p^{-1}(K'_\mu)$. We
have $p(U)=\cup_{\lambda\in L}K_\lambda$, $p(U')=\cup_{\mu\in
M}K'_\mu$, and $I=p(U)\cup p(U')$. Since $I$ is a connected
topological space, the non-empty open sets $p(U)$ and $p(U')$ are
not disjoint, that is, there exist indices $\lambda\in L$ and
$\mu\in M$ with $K_\lambda\cap K'_\mu\neq\emptyset$. We have
$J_\lambda\cap J'_\mu=p^{-1}(K_\lambda\cap K'_\mu)$ and the
surjectivity of $p$ implies $J_\lambda\cap J'_\mu\neq\emptyset$
--- a contradiction with $U\cap U'=\emptyset$. Therefore $A$ is
connected.

{\rm (iii)} Let $A$ be separable and let $D$ be a countable subset
of $A$ which meets every non-empty open interval $J$ in $A$. Then in
accord with Proposition~\ref{5.20.1}, {\rm (i)}, the countable image
$D'=p(D)$ meets any non-empty open interval $K$ in $I$.

Now, let $I$ be a separable topological space and let $D'\subset I$
be a countable subset such that $D'\cap K\neq\emptyset$ for any
non-empty open interval $K$ in $I$. Using the axiom of choice, we
fix an element $d\in p^{-1}(d')$ for any $d'\in D'$. The subset
$D\subset A$ of all $d\in A$ just chosen is countable and
Proposition~\ref{5.20.1}, {\rm (i)}, assures that $D$ meets every
non-empty open interval $J$ in $A$. Thus, $A$ is separable and,
moreover, $D'=p(D)$.

\end{proof}

\section{Characterization of the negative transitivity}

\label{2}

This section is the core of the paper. Here we give a complete
characterization of the sets of preferences $A$ whose preference
relation has negatively transitive asymmetric part. As a consequence
of this characterization we prove the existence of a generalized
continuous utility function on $A$ under the condition that the
topological space $A$ is connected and separable.

\subsection{The necessary condition}

\label{2.1}

\begin{theorem} \label{2.1.1} Let $A$ be a set of preferences
endowed with preference relation $R$ and let its asymmetric part
$F=P_R$ be negatively transitive. Then one has:

{\rm (i)} The relation $E_R$ is symmetric and ${\cal E}=E_F$ is an
equivalence relation on $A$.

{\rm (ii)} ${\cal E}=E\cup E_R$, where $E=I_R$ is the symmetric part
of $R$.

{\rm (iii)} The relation $F$ is ${\cal E}$-saturated.

{\rm (iv)} The relation $E$ is weakly ${\cal E}$-saturated.

{\rm (v)} The relation $R$ is weakly ${\cal E}$-saturated.

Let $A'=A/\!{\cal E}$ be the factor-set and let $R'$ and $\bar{F}$
be the weak factor-relation of $R$ and the factor-relation of $F$
with respect to the equivalence relation ${\cal E}$. Then one has:

{\rm (vi)} $R'=D_{A'}\cup\bar{F}$ and $D_{A'}\cap\bar{F}=\emptyset$.

{\rm (vii)} $R'$ is a reflexive, transitive, antisymmetric, and
complete binary relation on $A'$ with asymmetric part $\bar{F}$ and
the factor-set $A'$ endowed with $R'$ is a linearly ordered set.

{\rm (viii)} The inverse images $A_\iota=p^{-1}(\iota)$, $\iota\in
A'$, of the canonical surjective map $p\colon A\to A'$ are
$E$-saturated and any $A_\iota$ furnished with the equivalence
relation $E_\iota$ induced by $E$ is a balloon.

{\rm (ix)} The set of preferences $A$ is equal to the coproduct
$\coprod_{\iota\in A'}A_\iota$ of the family $(A_\iota)_{\iota\in
A'}$ of balloons.

\end{theorem}

\begin{proof} {\rm (i)} The relation $E_R$ is symmetric and
in accord with Proposition~\ref{5.1.65}, ${\cal E}$ is an
equivalence relation on $A$.

{\rm (ii)} We have ${\cal E}=(F\cup F^{-1})^c=F^c\cap
(F^c)^{-1}=(R^c\cup R^{-1})\cap ((R^c)^{-1}\cup R)=I_{R^c}\cup
I_R=E_R\cup I_R$.

{\rm (iii)} Since $F$ is asymmetric and negatively transitive,
Proposition~\ref{5.1.65} yields that $F$ is ${\cal E}$-saturated.

{\rm (iv)} This is a consequence of Lemma~\ref{5.1.1}.

{\rm (v)} Part {\rm (iii)} implies that $F$ is weakly ${\cal
E}$-saturated. Since $R=E\cup F$, we obtain that $R$ is weakly
${\cal E}$-saturated.

{\rm (vi)} Using parts {\rm (ii)}, {\rm (iii)}, and {\rm (iv)}, we
have $R'=E'\cup F'=D_{A'}\cup\bar{F}$, the last equality because of
Proposition~\ref{5.1.5}, {\rm (iv)}. If
$D_{A'}\cap\bar{F}\neq\emptyset$, then there exists $x\in A$ with
$\bar{x}\bar{F}\bar{x}$. In particular, $xFx$ and this contradicts
$E\cap F=\emptyset$. Therefore $D_{A'}\cap\bar{F}=\emptyset$.

{\rm (vii)} The weak factor-relation $R'$ is reflexive and
transitive because of Proposition~\ref{5.1.5}, {\rm (i)}, {\rm
(ii)}. Moreover, part {\rm (v)} together with
Proposition~\ref{5.1.5}, {\rm (iii)}, {\rm (iv)}, imply that $R'$ is
an antisymmetric relation. Now, let $\bar{x}\neq\bar{y}$. Then
$xE_F^cy$ and let us suppose, for example, that $xFy$. Therefore
$\bar{x}\bar{F}\bar{y}$ and this implies $\bar{x}R'\bar{y}$. Thus,
$R'$ is a complete relation. In accord with
Proposition~\ref{5.1.65}, $F$ is transitive and
Proposition~\ref{5.1.5}, {\rm (ii)}, assures that the
factor-relation $F'=\bar{F}$ is transitive, too. If
$\bar{x}\bar{F}\bar{y}$ and $\bar{y}\bar{F}\bar{x}$, then
$\bar{x}R'\bar{y}$ and $\bar{y}R'\bar{x}$, and the antisymmetry of
$R'$ implies $\bar{x}=\bar{y}$ --- a contradiction with part {\rm
(vi)}. Therefore the relation $\bar{F}$ on $A'$ is asymmetric and
the decomposition of $R'$ from part {\rm (vi)} yields
$P_{R'}=\bar{F}$.

{\rm (viii)} It is enough to note that $E\subset {\cal E}$.

{\rm (ix)} The set of preferences $A$ is equal to the coproduct
$\coprod_{\iota\in A'}A_\iota$ of the family $(A_\iota)_{\iota\in
A'}$ of sets. Since $R=E\cup F$ and since the relation $xFy$ is
equivalent to the inequality $p(x)<p(y)$ (part {\rm (vi)}), we have
$xRy$ if and only if ``$p(x)<p(y)$" or ``$p(x)=p(y)$ and
$xE_{p\left(x\right)}y$". In accord with Proposition~\ref{5.15.1},
{\rm (i)}, $A$ is the coproduct of the family $(A_\iota)_{\iota\in
A'}$ of balloons.

\end{proof}

\subsection{The sufficient condition}

\label{2.5}

According to Theorem~\ref{2.1.1}, {\rm (ix)}, any set of preferences
whose preference relation has negatively transitive asymmetric part
can be identified with the coproduct of a family of balloons with
linearly ordered index set. Theorem~\ref{2.5.5} below shows that
starting with a family of balloons indexed with a linearly ordered
set $I$, we can endow the coproduct $A$ of this family with a
preference relation $R$ such that its asymmetric part $F$ is
negatively transitive, the factor-order $R'$ on the factor-set
$A'=A/E_F$ is linear, and the linearly ordered set $A'$ is
isomorphic to $I$.

\begin{theorem}\label{2.5.5} Let $(A_\iota)_{\iota\in I}$ be a
family of balloons, let $(E_\iota)_{\iota\in I}$ be the family of
the corresponding equivalence relations, and let the index set $I$
be linearly ordered. Let ${\cal E}$ be the equivalence relation on
the coproduct $A=\coprod_{\iota\in I}A_\iota$, which corresponds to
the partition $(A_\iota)_{\iota\in I}$, let $p\colon A\to I$,
$x\mapsto\iota$ for $x\in A_\iota$, be the natural projection, and
let $R$ be the preference relation on $A$ from
Proposition~\ref{5.15.1}, {\rm (i)}, defined by the rule
\[
xRy \hbox{\rm\ if\ } "p(x) < p(y)" \hbox{\rm\ or\ } "p(x)=p(y)
\hbox{\rm\ and\ } xE_{p\left(x\right)}y".
\]

 {\rm (i)} The asymmetric part $F=P_R$ of $R$ is negatively
transitive and one has ${\cal E}=E_F$.

{\rm (ii)} The relation $R$ is weakly ${\cal E}$-saturated.

{\rm (iii)} The factor-set $A'=A/{\cal E}$ endowed with the
factor-relation $R'$ is isomorphic to the linearly ordered set $I$.

\end{theorem}

\begin{proof} {\rm (i)} Since
$E_\iota=I_{E_\iota}$ for any $\iota\in I$, using
Proposition~\ref{5.15.1}, {\rm (ii)}, we have
\[
xI_Ry \hbox{\rm\ if and only if\ }"p(x)=p(y) \hbox{\rm\ and\ }
xE_{p\left(x\right)}y"
\]
and this implies
\begin{equation}
xFy \hbox{\rm\ if and only if\ } p(x) < p(y).\label{2.5.10}
\end{equation}
In particular, $E_F=\{(x,y)\in A\mid p(x)=p(y)\}={\cal E}$.

Now, let $xFz$ and let $y\in A$. Then $p(x)<p(z)$ and the linearity
of $I$ yields $p(x)<p(y)$ or $p(y)<p(z)$, that is $xFy$ or $yFz$. In
other words, the asymmetric part $F$ of $R$ is negatively
transitive.

{\rm (ii)} Use Part {\rm (i)} and Theorem~\ref{2.1.1}, {\rm (v)}.

{\rm (iii)} Since $p$ is a surjective map and since ${\cal
E}=\{(x,y)\in A\mid p(x)=p(y)\}$, Proposition~\ref{5.1.5}, {\rm
(v)}, yields an increasing bijection $p'\colon A'\to I$ of linearly
ordered sets, that is, an isomorphism between $A'$ and $I$.

\end{proof}

\subsection{Existence of a continuous generalized utility function}

\label{2.10}

Let $\hbox{\ccc R}$ be the set of real numbers. Let $A$ be a set of
preferences, let $R$ be its preference relation with asymmetric part
$F$, and let $A$ be endowed with its interval topology. By a
\emph{generalized utility function} on $A$ we mean a real function
$u\colon A\to \hbox{\ccc R}$ such that: (a) $u(x) < u(y)$ if and
only if $xFy$; (b) $u(x)=u(y)$ if and only if $x{\cal E}y$, where
${\cal E}=E_F$ is the indifference ``$x$ and $y$ are not
$F$-comparable". The existence of a generalized utility function
yields that ${\cal E}$ is an equivalence relation and that $F$ is
negatively transitive. Conversely, in case $F$ is negatively
transitive, Theorem~\ref{2.1.1}, {\rm (i)}, {\rm (ii)}, assure that
${\cal E}$ is an equivalence relation and, moreover, under the
condition that $u$ exists, (b) is equivalent to: (b$'$) $u(x)=u(y)$
if and only if ``$xRy$ and $yRx$" or ``$x$ and $y$ are not
$R$-comparable". When the preference relation $R$ is complete,
condition (b$'$) has the form $u(x)=u(y)$ if and only if $xRy$ and
$yRx$, that is, $u$ is the ordinary utility function.

In the lemma below we use freely the terminology and notation from
Theorem~\ref{2.1.1} and Theorem~\ref{2.5.5}.

We denote by $U$ the set of the four intervals in $\hbox{\ccc R}$
with endpoints $0$ and $1$.

\begin{lemma}\label{2.10.1} Let $A$ be a set of preferences whose
preference relation $R$ has a negatively transitive asymmetric part
$F$, let $A$ be the coproduct of the corresponding family
$A_\iota)_{\iota\in I}$ of balloons whose index set $I$, after the
identification $I=A/{\cal E}$, is furnished with the linear
factor-order, and let $p\colon A\to I$, $x\mapsto\iota$ for $x\in
A_\iota$, be the natural projection. Let $A$ and $I$ be endowed with
their interval topologies and let $\Lambda\in U$. The following two
respective statements are then equivalent:

{\rm (i)} There exists a surjective generalized (and continuous)
utility function $u\colon A\to \Lambda$.

{\rm (ii)} There exists a strictly increasing surjective (and
continuous) map $u'\colon I\to \Lambda$.

Parts {\rm (i)}, {\rm (ii)} with continuous $u$, $u'$ imply:

{\rm (iii)} There exists a homeomorphism $u'\colon I\to \Lambda$.

If, in addition, $I$ is connected, then parts {\rm (i)}, {\rm (ii)}
with continuous $u$, $u'$ and part {\rm (iii)} are equivalent.

\end{lemma}

\begin{proof}
{\rm (i)} $\Longleftrightarrow$ {\rm (ii)} There exists a bijection
between the maps $u\colon A\to \hbox{\ccc R}$ which are constant on
the members of the factor-set $I=A/{\cal E}$ and the maps $u'\colon
I\to \hbox{\ccc R}$ such that the diagram
\begin{diagram}
A           &         & \\
\dTo_{p} & \rdTo^{u}& \\
I        & \rTo^{u'} & \hbox{\ccc R}\\
\end{diagram}
is commutative. The equivalence~(\ref{2.5.10}) shows that $u$ is
strictly increasing if and only if $u'$ is strictly increasing.
Moreover, since ${\cal E}=\{(x,y)\in A^2\mid u(x)=u(y)\}$,
Proposition~\ref{5.1.5}, {\rm (v)}, implies that the map $u'$ is a
strictly increasing bijection. Proposition~\ref{5.20.1}, {\rm
(iii)}, and the commutativity of the above diagram ($u'=p\circ u$)
assure that $u$ is continuous if and only if $u'$ is continuous.

{\rm (ii)} $\Longrightarrow$ {\rm (iii)} Let $u'\colon I\to \Lambda$
be a strictly increasing surjective continuous map. Then $u'$ is an
isomorphism of posets and maps any open interval $K$ in $I$ onto the
open interval $u'(K)$ in $\Lambda$. Thus, the bijection $u'$ is a
continuous and open map, hence its inverse $u'^{-1}\colon \Lambda\to
I$ is a continuous map, too.

{\rm (ii)} $\Longleftrightarrow$ {\rm (iii)} Let  $u$ and $u'$ be
continuous. Under the condition of connectedness of $I$, this
equivalence may be found in~\cite[Ch. IV, Sec. 2, Exercise 8
a]{[5]}.

\end{proof}

 Let $\hbox{\ccc Q}$ be the set of rational numbers.

\begin{theorem}\label{2.10.5} Let $A$ be a set of preferences whose
preference relation has a negatively transitive asymmetric part. Let
$A$ be connected and has a countable dense subset $D$ with respect
to the interval topology on $A$. Then there exists a continuous
generalized utility function $u$ on $A$ which maps the set $A$ onto
some interval $\Lambda\in U$ and its subset $D$ onto the
intersection $\Lambda\cap\hbox{\ccc Q}$.

\end{theorem}

\begin{proof} Under the conditions of Lemma~\ref{2.10.1},
Corollary~\ref{5.20.15}, {\rm (ii)}, {\rm (iii)}, yields that the
index set $I$ is connected and separable with $D'=p(D)$ as a
countable dense subset. Now, in accord with~\cite[Ch. IV, Sec. 2,
Exercise 11 a]{[5]}, there exist a interval $\Lambda\in U$ and a map
$u'\colon I\to \Lambda$ which is an increasing homeomorphism of $I$
onto $\Lambda$ and maps $D'$ onto the intersection
$\Lambda\cap\hbox{\ccc Q}$. Composing $u'$ with the natural
projection $p\colon A\to I$, we obtain a map $u\colon A\to\Lambda$
which according to Lemma~\ref{2.10.1} is the desired continuous
generalized utility function.

\end{proof}

\begin{remark} \label{2.10.10} {\rm In the case of plane
$A=\hbox{\ccc R}^2$ endowed with the lexicographical order and the
interval topology, the connectedness of the topological space $A$
fails to be true because the connected components of $A$ are exactly
the "vertical lines" $L_a=\{(a,b)\in\hbox{\ccc R}^2\mid b\in
\hbox{\ccc R}\}$, $a\in \hbox{\ccc R}$. Moreover, it is well known
that on the (complete) set of preferences $A$ there is no utility
function --- see, for example,~\cite[Chapter B, Section 4, 4.2,
Example 1]{[32]}. Thus, we can not expect a general result without
$A$ being connected.

}

\end{remark}

\appendix

\section{Appendix}

\subsection{Partially Ordered sets: Miscellaneous Results}

Let $A$ be a set. The set ${\cal PO}(A)$ consisting of all partial
orders $R\subset A^2$ on $A$ with the relation $R\subset R^\prime$
(that is, inclusion of partial orders) is a poset. Any linear order
$R$ on $A$ is a maximal element of the poset ${\cal PO}(A)$ because
if $(x,y)\notin R$, then $(y,x)\in R$.

\begin{lemma}\label{0.7.2} If $R$ is a non-linear partial order
on the set $A$ and if $a,b\in A$ is a pair such that $(a,b)\notin R$
and $(b,a)\notin R$, then there exists a partial order $R'$ which
extends $R$ and contains $(a,b)$.

\end{lemma}

\begin{proof} We set
\[
R_{\left(a,b\right)}=\{(x,y)\in A^2\mid (x,a)\in R\hbox{\rm\ and\ }
(b,y)\in R\},
\]
and $R'=R\cup R_{\left(a,b\right)}$. Since $(a,b)\in
R_{\left(a,b\right)}$ we have $R'\neq R$. The relation $R'$ is
reflexive since $R$ is reflexive.

Now, let $(x,y)\in R'$ and $(y,x)\in R'$. The relations $(x,y)\in
R_{\left(a,b\right)}$ and $(y,x)\in R_{\left(a,b\right)}$ mean
$(x,a)\in R$, $(b,y)\in R$, $(y,a)\in R$, and $(b,x)\in R$, which in
turn implies $(b,a)\in R$
--- a contradiction. If, for example, $(x,y)\in R$ and $(y,x)\in
R_{\left(a,b\right)}$, then $(y,a)\in R$ and $(b,x)\in R$. We have
$(x,a)\in R$ and this yields $(b,a)\in R$ --- again a contradiction.
Thus, the relation $R'$ is antisymmetric.

Let $(x,y)\in R'$ and $(y,z)\in R'$. The case $(x,y)\in
R_{\left(a,b\right)}$ and $(y,z)\in R_{\left(a,b\right)}$ is
impossible because $(b,y)\in R$ and $(y,a)\in R$ imply $(b,a)\in R$:
a contradiction. Now, let $(x,y)\in R$ and $(y,z)\in
R_{\left(a,b\right)}$. Then $(y,a)\in R$, $(b,z)\in R$, $(x,a)\in
R$, and this yields $(x,z)\in R_{\left(a,b\right)}$. In the case
$(x,y)\in R_{\left(a,b\right)}$ and $(y,z)\in R$ we have $(x,a)\in
R$, $(b,y)\in R$, $(b,z)\in R$, hence $(x,z)\in R$.

\end{proof}

We obtain immediately

\begin{corollary}\label{0.7.3} Any non-linear partial order on a
set is not maximal.

\end{corollary}

\begin{theorem}\label{0.9.5} {\rm (E.~Szpilrajn)} Let $A$ be set.
Any partial order on $A$ can be extended to a linear order on $A$.

\end{theorem}

\begin{proof} Let $R$ be a partial order on $A$ and let
us consider the non-empty set ${\cal E}(R)\subset {\cal PO}(A)$
consisting of all partial orders that extend $R$.  The partial order
in ${\cal PO}(A)$ induces a partial order on ${\cal E}(R)$ and let
$C\subset {\cal E}(R)$ be a chain. The union U=$\cup_{S\in C}S$ is a
partial order that extends $R$, that is, $U\in {\cal E}(R)$, and,
moreover, $U$ is an upper bound of $C$. Therefore ${\cal E}(R)$ is
an inductively ordered poset and Kuratowski-Zorn theorem yields that
there exists a maximal element $M\in {\cal E}(R)$. If the order $M$
is not linear, then Lemma~\ref{0.7.2} produces a contradiction.

\end{proof}

Let $A$ and $A'$ be posets. \emph{Isomorphism of posets} is a
bijection $f\colon A\to A'$ such that the relations $x\leq y$ and
$f(x)\leq f(y)$ are equivalent.

\begin{proposition}\label{0.7.10} Let $A$ be a linearly ordered
set, and let $B$ be a poset. Every strictly monotonic map $f\colon
A\to B$ is injective; if $f$ is strictly increasing, then $f$ is an
isomorphism of $A$ and $f(A)$.

\end{proposition}

\begin{proof} Indeed, $x\neq y$ implies $x<y$ or $x>y$,
hence $f(x)<f(y)$ or $f(x)>f(y)$, and in all cases $f(x)\neq f(y)$.
Finally, $f(x)\leq f(y)$ implies $x\leq y$: otherwise we would have
$x>y$ and then $f(x)>f(y)$: a contradiction.

\end{proof}

\begin{proposition}\label{0.2.20} Let $A$ and $B$ be posets and
let $f\colon A\to B$ be a bijection. Then $f$ is an isomorphism of
posets if and only if $f$ and its inverse $f^{-1}$ are increasing
maps, and under this condition $f$ and $f^{-1}$ are strictly
increasing maps.

\end{proposition}

\begin{proof} The ``only if" part is immediate. Now, let
$f$ and its inverse $f^{-1}$ be increasing maps. The relation $x\leq
y$ in $A$ implies $f(x)\leq f(y)$ in $B$. The relation $x'\leq y'$
in $B$ implies $f^{-1}(x')\leq f^{-1}(y')$ in $A$, and it is enough
to set $x'=f(x)$ and $y'=f(y)$. Now, let $f$ be an isomorphism of
posets. Then $x<y$ implies $f(x)\leq f(y)$ and if $f(x)=f(y)$, then
$x=y$: a contradiction. Therefore $f(x) < f(y)$.

\end{proof}

\begin{lemma}\label{0.2.191} Let $A$ be a linearly preordered set.
Then the intersection of any finite family of open intervals in $A$
is an open interval in $A$.

\end{lemma}

\begin{proof} Let $((a_\iota,b_\iota))_{\iota\in I}$ be a finite
family of open intervals. Let $a$ be the greatest element of the
finite family $(a_\iota)_{\iota\in I}$ and let $b$ be the least
element of the finite family $(b_\iota)_{\iota\in I}$. Note that
$a=\leftarrow$ if $a_\iota=\leftarrow$ for all $\iota\in I$ and
$b=\rightarrow$ if $b_\iota=\rightarrow$ for all $\iota\in I$. Then
we have
\[
\cap_{\iota\in I}(a_\iota,b_\iota)=(a,b).
\]

\end{proof}

\begin{theorem}{\rm (G.~Cantor)}\label{0.9.25} Let $A$ and
$B$ be countable linearly ordered sets that have least elements and
greatest elements.

{\rm (i)} If $B$ is without gaps, then there exists a strictly
increasing map $f\colon A\to B$.

{\rm (ii)} If $A$ and $B$ are without gaps, then the map $f$ from
part {\rm (i)} is an isomorphism of posets.

\end{theorem}

\begin{proof} {\rm (i)} We set $A=\{a_1,a_2,\ldots, \}$ and
$B=\{b_1,b_2,\ldots,\}$ and without any loss of generality we can
suppose that $a_1$, $b_1$ are the least and $a_2$, $b_2$ are the
greatest elements of $A$ and $B$. Further, we set
$A_n=\{a_1,a_2,\ldots, a_n\}$, $n\geq 2$. We define inductively a
sequence of strictly increasing maps $(f_n\colon A_n\to B)_{n\geq
2}$ such that for any $n$ the map $f_n$ is an extension of
$f_{n-1}$. First we define a strictly increasing map $f_2\colon
A_2\to B$ by $f_2(a_1)=b_1$ and $f_2(a_2)=b_2$. Given $n\geq 3$,
under the assumption that $f_{n-1}$ is a strictly increasing map, we
define $f_n$ via the rule
\begin{equation}
f_n(a_n)=b_s,\label{0.9.30}
\end{equation}
where $s\geq 3$ is the minimal index such that $f_n$ is a strictly
increasing map. It is enough to show that there exists an index
$s\geq 3$ that satisfies~(\ref{0.9.30}) and has the latter property.
Let $a_i$ be the immediate predecessor and $a_j$ be the immediate
successor of $a_n$ in the finite poset $A_n$. We have
$f_{n-1}(a_i)<f_{n-1}(a_j)$ and since the open interval
$I=(f_{n-1}(a_i),f_{n-1}(a_j))$ is not empty, each $b_s\in I$ works.
The sequence $(f_n)_{n\geq 2}$, in turn, defines a strictly
increasing map $f\colon A\to B$ by the formula $f(a_n)=f_m(a_n)$ for
some $m\geq n$.

{\rm (ii)} It is enough to prove that for any integer $p\geq 2$ we
have $b_k\in {\rm Im} (f)$ for all $k<p$. This statement is true for
$p=2$. Now, let $p\geq 3$, let us suppose that the statement is true
for all integers $< p$, and let us choose $n\geq 2$ such that
$b_1,\ldots, b_{p-1}\in f_n(A_n)$. Let us suppose that $b_p\notin
f_n(A_n)$, let $b_\iota$ be the immediate predecessor, $b_\kappa$ be
the immediate successor of $b_p$ in the finite poset $f_n(A_n)$, and
let $f_n(a_i)=b_\iota$, $f_n(a_j)=b_\kappa$. We have $A_n\cap
(a_i,a_j)=\emptyset$ and choose $m$ to be the minimal index such
that $a_m\in (a_i,a_j)$. Then $m>n$, $A_m\cap (a_i,a_j)=\{a_m\}$,
and, in particular, $a_i$ is the immediate predecessor and $a_j$ is
the immediate successor of $a_m$ in the finite poset $A_m$.
Moreover, $p$ is the minimal index with $b_p\in
(f_n(a_i),f_n(a_j))$. In accord with the definition of the map $f$
from part {\rm (i)}, we have $f(a_m)=f_m(a_m)=b_p$. Thus, our
statement is true for $p+1$ and the principle of mathematical
induction yields it for all $p\geq 2$.

\end{proof}

\begin{corollary}\label{0.9.35} For any countable
linearly ordered set $A$ there exists a strictly increasing map
$f\colon A\to [0,1]_{\hbox{\cccc Q}}$. If, in addition, $A$ is
without gaps, then:

{\rm (i)} The map $f$ establishes an isomorphism between $A$ and the
interval $[0,1]_{\hbox{\cccc Q}}$ if and only if $A$ has a least
element and a greatest element.

{\rm (ii)} The map $f$ establishes an isomorphism between $A$ and
the interval $[0,1)_{\hbox{\cccc Q}}$ if and only if $A$ has least
element but has no greatest element.

{\rm (iii)} The map $f$ establishes an isomorphism between $A$ and
the interval $(0,1]_{\hbox{\cccc Q}}$ if and only if $A$ has a
greatest element but has no least element.

{\rm (iv)} The map $f$ establishes an isomorphism between $A$ and
$(0,1)_{\hbox{\cccc Q}}$ if and only if $A$ has no least element and
no greatest element.

\end{corollary}

\begin{proof} The existence of a strictly
increasing map $f\colon A\to [0,1]_{\hbox{\cccc Q}}$ is ensured by
Theorem~\ref{0.9.25}, {\rm (i)}. The necessity part of the
equivalences in {\rm (i)} -- {\rm (iv)} is immediate.

{\rm (i)} Note that the interval $[0,1]_{\hbox{\cccc Q}}$ has least
element $0$, greatest element $1$, has no gaps, and then use
Theorem~\ref{0.9.25}, {\rm (ii)}.

{\rm (ii)} We adjoint a greatest element $\gamma$ to $A$, thus
obtaining $A'=A\oplus\{\gamma\}$, use part {\rm (i)} for $A'$, and
then restrict the corresponding isomorphism on $A$.

{\rm (iii)} We adjoint a least element $\lambda$ to $A$, thus
obtaining $A'=\{\lambda\}\oplus A$, use part {\rm (i)} for $A'$, and
then restrict the corresponding isomorphism on $A$.

{\rm (iv)} We adjoint a greatest element $\gamma$ and a least
element $\lambda$ to $A$, thus obtaining $A''=\{\lambda\}\oplus
A\oplus\{\gamma\}$, use part {\rm (i)} for $A''$, and then restrict
the corresponding isomorphism on $A$.

\end{proof}

Using Corollary~\ref{0.9.35}, {\rm (iv)}, twice we obtain

\begin{corollary}\label{0.9.40} Every countable subset of
the open interval $(0,1)_{\hbox{\cccc R}}$, which is dense in this
interval, is isomorphic to $\hbox{\ccc Q}$.

\end{corollary}

Combining Theorem~\ref{0.9.5} and Corollary~\ref{0.9.35} we obtain

\begin{corollary}\label{0.9.45} Given a countable partially
ordered set $A$, there exists a strictly increasing map of $A$ into
$\hbox{\ccc Q}$.

\end{corollary}

\begin{corollary}\label{0.9.50} If $A$ is a countable linearly
ordered set $A$ without gaps, then there exists a bounded from above
subset $B$ of $A$, which has no least upper bound in $A$.

\end{corollary}

\begin{proof}  Let $f\colon A\to [0,1]_{\hbox{\cccc Q}}$ be a
strictly increasing map with image $I$ which coincides with one of
the four intervals in the rational line $\hbox{\ccc Q}$ with
endpoints $0$ and $1$. The existence of $f$ is assured by
Corollary~\ref{0.9.35}. Let $J=I\cap
(0,\frac{\sqrt{2}}{2})_{\hbox{\cccc R}}$ and let $B=f^{-1}(J)$.
Since the subset $J$ of $I$ has no least upper bound in $I$, the
subset $B$ of $A$ has no least upper bound in $A$.

\end{proof}

\subsection{Interval Topology: Connectedness}

\label{1.8}

Let $A$ be a preordered set with preorder $R$ and let $A$ be
furnished with its interval topology ${\cal T}_0(A)$. A closed
interval $[x,y]$ with $xRy$ is said to be a \emph{gap} in $A$ if
$[x,y]=\{x,y\}$, that is, the open interval $(x,y)$ is empty.

We use the same terminology and notation for a partially ordered set
$A$.

\begin{examples} \label{1.8.400} {\rm (1) The natural topology of
the real line $\hbox{\ccc R}$ coincides with its interval topology
${\cal T}_0(\hbox{\ccc R})$. The (countable) set of all open
intervals with rational endpoints is a base of this topology and the
trace of this set on $\hbox{\ccc Q}\subset \hbox{\ccc R}$ is by
definition the natural base of the interval top[ology ${\cal
T}_0(\hbox{\ccc Q})$ on the rational line $\hbox{\ccc Q}$. In
particular, the topology of $\hbox{\ccc Q}$ as a subspace of the
real line $\hbox{\ccc R}$ coincides with its interval topology
${\cal T}_0(\hbox{\ccc Q})$.

(2) The natural topology of any interval $I$ in the real line
$\hbox{\ccc R}$, that is, the topology on $I$, considered as a
subspace of $\hbox{\ccc R}$, coincides with its interval topology
${\cal T}_0(I)$ because the trace of any open interval in
$\hbox{\ccc R}$ on $I$ is an open interval in $I$.

(3) The natural topology of any interval $J$ in the rational line
$\hbox{\ccc Q}$, that is, the topology on $J$, considered as a
subspace of $\hbox{\ccc Q}$, coincides with its interval topology
${\cal T}_0(J)$ because the trace of any open interval in
$\hbox{\ccc Q}$ on $J$ is an open interval in $J$. In accord with
Example (1), any open interval $J$ in the rational line $\hbox{\ccc
Q}$ with its natural topology is a subspace of the real line
$\hbox{\ccc R}$, too.

(4) Given a natural number $n\in\hbox{\ccc N}$, we set
$A=[0,1]_{\hbox{\cccc Q}}$,
\[
F_n=\{(\frac{1}{i},\frac{1}{j})\in\hbox{\ccc Q}\times \hbox{\ccc
Q}\mid i,j\in\hbox{\ccc N}, 1\leq j < i\leq
n\}\cup\{(0,\frac{1}{i})\mid i\in\hbox{\ccc N}, 1\leq i\leq n\}.
\]
The set $F_n$ is an asymmetric and transitive binary relation on $A$
and $R_n=D_A\cup F_n$ is a reflexive and transitive binary relation
on $A$, that is, a partial order on $A$, with asymmetric part $F_n$.
We denote by $A_n$ the set $A$ endowed with the partial order $R_n$.
Every member of $F_n$, that is, an ordered pair of the form
$(\frac{1}{i},\frac{1}{j})$, $i,j\in\hbox{\ccc N}$, $1\leq j < i\leq
n$, or $(0,\frac{1}{i})$, $i\in\hbox{\ccc N}$, $1\leq i\leq n$,
defines an open interval in the interval topology ${\cal T}_0(A_n)$:
\[
(\frac{1}{i},\frac{1}{j})_{A_n}=\{a\in A\mid (\frac{1}{i},a)\in
F_n\hbox{\rm\ and\ } (a,\frac{1}{j})\in F_n\},
\]
and
\[
(0,\frac{1}{i})_{A_n}=\{a\in A\mid (0,a)\in F_n\hbox{\rm\ and\ }
(a,\frac{1}{i})\in F_n\},
\]
respectively. Every element of the family $\hbox{\ccc F}_n$ of open
intervals thus obtained has the form
\[
(\frac{1}{i},\frac{1}{j})_{A_n}=\{\frac{1}{k}\mid \frac{1}{i}<
\frac{1}{k}<\frac{1}{j}\hbox{\rm\ and\ }k\in\hbox{\ccc N}, 1\leq
k\leq n\},
\]
and
\[
(0,\frac{1}{i})_{A_n}=\{\frac{1}{k}\mid
\frac{1}{k}<\frac{1}{i}\hbox{\rm\ and\ }k\in\hbox{\ccc N}, 1\leq
k\leq n\}.
\]
In particular, the closed intervals $[0,\frac{1}{n}]_{A_n}$ and
$[\frac{1}{j+1},\frac{1}{j}]_{A_n}$, $1\leq j\leq n-1$, are all gaps
in the partially ordered set $A_n$.

Since the intersection of any finite number of members of the family
$\hbox{\ccc F}_n$ is again a member of $\hbox{\ccc F}_n$, we obtain
that the finite family $\hbox{\ccc F}_n$ is a base of the interval
topology ${\cal T}_0(A_n)$. On the other hand, any union of members
of $\hbox{\ccc F}_n$ equals a finite union of disjoint members of
$\hbox{\ccc F}_n$. Therefore the interval topology ${\cal T}_0(A_n)$
is a finite set consisting of finite unions of disjoint members of
$\hbox{\ccc F}_n$. In particular, all open sets different from $A_n$
are finite and all closed sets different from the empty set
$\emptyset$ are infinite (and countable), which implies that $A_n$
is a connected topological space.

}

\end{examples}

\begin{proposition}\label{1.8.4} Let $A$ and $A'$ be partially
ordered sets endowed with interval topology.

{\rm (i)} If $f\colon A\to A'$ is an isomorphism of posets, then $f$
is a homeomorphism.

{\rm (ii)} Let $A$ be a linearly ordered set and $f\colon A\to A'$
be a strictly increasing map with image $C\subset A'$. If the
topology of $C$ induced from $A'$ coincides with its interval
topology, then $f$ is a homeomorphism of $A$ onto the chain $C$
considered as a subspace of $A'$.

\end{proposition}

\begin{proof} {\rm (i)} In accord with Proposition~\ref{0.2.20},
$f$ and its inverse $f^{-1}$ are strictly increasing maps. Therefore
\[ f((x,y))=(f(x),f(y)),\hbox{\rm\  and\ }
f^{-1}((x',y'))=(f^{-1}(x'),f^{-1}(y'))
\]
for any open interval $(x,y)\subset A$ and any open interval
$(x',y')\subset A'$. Thus, the map $f$ (respectively, $f^{-1}$) maps
the sub(base) of the topological space $A$ (respectively, $A'$) onto
the (sub)base of the topological space $A'$ (respectively, $A$) and
therefore $f$ is a homeomorphism.

{\rm (ii)} According to Proposition~\ref{0.7.10}, $f\colon A\to C$
is an isomorphism of posets. If $U'\subset A'$ is open, then the
trace $U'\cap C$ is open in $C$ and hence $f^{-1}(U')=f^{-1}(U'\cap
C)$ is an open set in $A$ because of part {\rm (i)}.

\end{proof}

The map
\[
f\colon \hbox{\ccc R}\to (-1,1)_{\hbox{\cccc R}} ,\hbox{\
}x\mapsto\frac{x}{1+|x|},
\]
is a strictly increasing bijection and the map
\[
g\colon (-1,1)_{\hbox{\cccc R}}\to (0,1)_{\hbox{\cccc R}} ,\hbox{\
}x\mapsto 2\frac{x+1}{x+3},
\]
is a strictly increasing bijection. Thus, the composition $h=g\circ
f$,
\[
h\colon \hbox{\ccc R}\to (0,1)_{\hbox{\cccc R}} ,\hbox{\
}x\mapsto\frac{x+1+|x|}{x+3+3|x|},
\]
is a strictly increasing bijection and in accord with
Proposition~\ref{0.7.10}, $h$ is an isomorphism of posets. The
restriction
\[
h_{\mid\hbox{\cccc Q}}\colon \hbox{\ccc Q}\to (0,1)_{\hbox{\cccc Q}}
,\hbox{\ }x\mapsto\frac{x+1+|x|}{x+3+3|x|},
\]
is a strictly increasing bijection and the same argument yields that
it is an isomorphism of posets.

\begin{proposition} \label{1.8.390} Both bijections $h$ and
$h_{\mid\hbox{\cccc Q}}$ are homeomorphisms of the corresponding
sets endowed with the interval topology.

\end{proposition}

\begin{proof} We apply Proposition~\ref{1.8.4}.

\end{proof}

\begin{remarks}\label{1.8.3} {\rm (1) We note that the base of
the interval topology on the underlying set $A$ is invariant if we
replace the poset $A$ with its dual $A^{op}$. Therefore the dual
poset structure produces the same interval topology ${\cal T}_0(A)$.

(2) If the subset $B\subset A$ is endowed with the induced linear
order, then the corresponding interval topology ${\cal T}_0(B)$ is,
in general, weaker than the topology of $B$ considered as a subspace
of $A$.}

\end{remarks}

\begin{proposition}\label{1.8.5} Let $A$ be a linearly
ordered set endowed with the interval topology and let $B\subset A$,
$x,y\in A$, and $x < y$. If $A$ is a set without gaps and
$(x,y)\subset B$, then $x\in \bar{B}$ and $y\in \bar{B}$, where
$\bar{B}$ is the closure of $B$ in $A$.

\end{proposition}

\begin{proof}  Let $I$ be an open interval that contains $x$
(respectively, $y$). It is enough to consider intervals of the form
$I=(a,b)$ with $a < x < b$ (respectively, $a < y < b$). We set
$m=\min \{y,b\}$ (respectively, $M=\max \{x,a\}$). Since $A$ is a
set without gaps, the open interval $(x,m)$ (respectively,$(M,y)$)
is not empty and $(x,m)\subset B\cap I$ (respectively, $(M,y)\subset
B\cap I$). In particular, $B\cap I\neq\emptyset$ and hence $x\in
\bar{B}$ (respectively, $y\in \bar{B}$).

\end{proof}

\begin{corollary}\label{1.8.10} If $z=\inf B$ or $z=\sup B$,
then $z\in \bar{B}$.

\end{corollary}

\begin{proof}  Let $z=\inf B$ (respectively, $z=\sup B$) and
let $I=(a,b)$ be an open interval that contains $z$. Then there
exists $x\in B$ with $z < x < b$ (respectively, $a < x < z$) and, in
particular, $I\cap B\neq\emptyset$. Thus, $z\in \bar{B}$.

\end{proof}

\begin{proposition}\label{1.8.16} Let $A$ be a non-empty linearly
ordered set equipped with the interval topology. The topological
space $A$ is compact if and only if any subset of $A$ has a least
upper bound and a greatest lower bound.

\end{proposition}

\begin{proof} Let $A$ be a compact space and let $B$ be a non-empty
subset of $A$ endowed with the induced linear order. Since $B$ is
filtered to the right, we can form the section filter of $B$, that
is, the filter on $B$ of base
\[
{\cal B}=\{B_x\mid B_x=[x,\rightarrow)_B,\hbox{\ }x\in B\}
\]
consisting of all closed right-unbounded intervals in $B$. Let
${\cal G}$ be the filter generated by ${\cal B}$ when ${\cal B}$ is
considered as a filter base on $A$. Note that the filter ${\cal G}$
is finer than the filter on $A$, generated by the filter base
\[
{\cal A}=\{A_x\mid A_x=[x,\rightarrow)_A,\hbox{\ }x\in B\}.
\]
Since $A$ is compact, the filter ${\cal G}$ has a cluster point
$a\in A$. Therefore $a$ is a cluster point of the filter base ${\cal
A}$ which consists of closed subsets of $A$; in particular $a\in
A_x$ for all $x\in B$, that is, $a$ is an upper bound of the set
$B$. If $b$ is an upper bound of $B$ and $b<a$, then the open
neighborhood $(b,\rightarrow)$ of the point $a$ does not contains
elements of $B$, which is a contradiction. Hence $a\leq b$ for any
upper bound $b$ of $B$, that is, $a=\sup_AB$. The topological space
$A$ is still compact if we change the structure of the linearly
ordered set on $A$ with its dual $A^{op}$. Then the latter statement
means that any non-empty subset of $A^{op}$ has a least upper bound,
that is, any non-empty subset of $A$ has a greatest lower bound. In
particular, the whole set $A$ has a least element $m$ and a greatest
element $M$. If $B=\emptyset$, then $\sup_AB=m$ and $\inf_AB=M$.

Conversely, let us suppose that any non-empty subset of $A$ has a
least upper bound and a greatest lower bound and let ${\cal F}$ be a
filter on $A$. Let $L$ be the set of all greatest lower bounds
$\inf_AC$ of members $C$ of the filter ${\cal F}$ and let
$a=\sup_AL$. We will prove that $a$ is a cluster point of ${\cal
F}$. Let us suppose that there exists $C\in {\cal F}$ such that
$\sup_AC < a$. Then there exists $C'\in {\cal F}$ with $\sup_AC <
\inf_AC'$, hence
\[
\sup_AC\cap C'\leq\sup_AC < \inf_AC'\leq\inf_AC\cap C',
\]
which is a contradiction because $C\cap C'\in {\cal F}$ and, in
particular, $C\cap C'\neq\emptyset$. Thus, for any $C\in {\cal F}$
we have $a\in [\inf_AC,\sup_AC]$. If $\inf_AC_0=\sup_AC_0$ for some
$C_0\in {\cal F}$, then $C_0=\{x\}$, $x\in A$, and the filter ${\cal
F}$ is the trivial ultrafilter $U(x)$. Since $x\in C$ for all $C\in
{\cal F}$, we have $\inf_AC\leq x$ for all $C\in {\cal F}$.
Therefore $a=\sup_AL\leq x$. On the other hand, $\{x\}\in {\cal F}$,
so $x=\inf_A\{x\}\leq a$, hence $x=a$. Since the ultrafilter ${\cal
F}=U(a)$ converges to $a$, the point $a$ is a cluster point of
${\cal F}$. Now, let us consider the case when $\inf_AC <\sup_AC$
for all $C\in {\cal F}$. Let us suppose the existence of an element
$C\in {\cal F}$ and an open interval $I$ that contains $a$, such
that $C\cap I=\emptyset$. If $I=(a',\rightarrow)$ (respectively,
$I=(\leftarrow,a'')$) then $\sup_AC\leq a' < a$ (respectively, $a <
a''\leq\inf_AC$), which in both cases is a contradiction. Now,
suppose that $I=(a',a'')$ where $a' < a''$. Then $C=C'\cup C''$,
where $C'=C\cap (\leftarrow,a']$ and $C''=C\cap [a'',\rightarrow)$.
Since $a' < a$, there exists $D\in {\cal F}$ such that $a' < \inf
D$. Now, the inequalities $\inf_AD\leq a < a''$ yield $C\cap
D=C''\cap D$ and, in particular, $\inf_A(C\cap D)\geq a'' > a$. On
the other hand, $C\cap D\in {\cal F}$ implies $\inf_A(C\cap D) \leq
a$ which is a contradiction. Thus, $a$ is a cluster point of the
filter ${\cal F}$.

\end{proof}

\begin{proposition}\label{1.8.20} Let $A$ be a non-empty linearly
ordered set endowed with the interval topology. If every closed
interval $[x,y]$, $x,y\in A$, $x < y$, is a connected subset of $A$,
then the topological space $A$ is connected.

\end{proposition}

\begin{proof} If $A=X\cup Y$ where $X$ and $Y$ are non-empty
disjoint open subsets of $A$ and if $x\in X$, $y\in Y$ with $x < y$,
then the closed interval $I=[x,y]$ is not connected. Indeed,
$I=(I\cap X)\cup (I\cap Y)$, where the traces $I\cap X$, $I\cap Y$
are non-empty, disjoint, and open (with respect to the induced
topology) sets of $I$.

\end{proof}

\begin{proposition}\label{1.8.25} Let $A$ be a non-empty linearly
ordered set endowed with the interval topology and let $B$ be a
non-empty subset of $A$, bounded from above. Let $U$ be the set of
all upper bounds of $B$ and $L$ be the set of all lower bounds of
$U$. Then one has:

{\rm (i)} $L$ and $U$ are non-empty subsets of $A$ with $A=L\cup U$.

{\rm (ii)} The following three statements are equivalent:

{\rm (a)} $z\in L\cap U$;

{\rm (b)} $z$ is the greatest element of $L$;

{\rm (c)} $z=\sup B$.

{\rm (iii)} If the intersection $L\cap U$ is empty, then $L$ and $U$
are open sets and $A$ is not connected.

\end{proposition}

\begin{proof} {\rm (i)} We have $B\subset L$ and, in particular,
$L\neq\emptyset$. Since $B$ is bounded from above, $U\neq\emptyset$.
If $x\in A$ then either $x\in U$, or there exists $b\in B$ with $x <
b$ and this implies $x < y$ for all $y\in U$, hence $x\in L$.

{\rm (ii)} {\rm (a)} $\Longrightarrow$ {\rm (b)} and {\rm (c)}. The
members $z\in L\cap U$ satisfy $z\in L$, $x\leq z$ for all $x\in L$,
and $z\in U$, $z\leq y$ for all $y\in U$. In other words, $z$ is the
greatest element of $L$ and the least element of $U$.

{\rm (b)} or {\rm (c)} $\Longrightarrow$ {\rm (a)}. Let $z$ be the
greatest element of $L$, that is, $x\leq z$ for all $x\in L$ and
$z\leq y$ for all $y\in U$. In particular, $z\in U$. Finally, let
$z=\sup B$, that is, $z\in U$ and $z\leq y$ for all $y\in U$. In
particular, $z\in L$.

{\rm (iii)} Let $L\cap U=\emptyset$. Part {\rm (i)} implies that the
subset $B$ has no least upper bound, that is, for any $y\in U$ there
exists $y'\in U$ with $y\in (y',\rightarrow)$. Thus $U$ is an open
set. Because of part {\rm (ii)} the subset $L$ has no greatest
element. Therefore for any $x\in L$ there exists $x'\in L$ with
$x\in (\leftarrow, x')$. In other words, $L$ is an open set too and
because of part {\rm (i)}, the topological space $A$ is not
connected.

\end{proof}

\begin{proposition}\label{1.8.30} Let $A$ be a non-empty
linearly ordered set endowed with the interval topology. If the
topological space $A$ is connected, then:

{\rm (i)} Every non-empty bounded from above subset $B\subset A$ has
a least upper bound.

{\rm (ii)} $A$ is a set without gaps.

\end{proposition}

\begin{proof} {\rm (i)} Let $U$ be the set of
all upper bounds of $B$ and $L$ be the set of all lower bounds of
$U$. Proposition~\ref{1.8.25}, {\rm (iii)}, implies that the
intersection $L\cap U$ is not empty and then
Proposition~\ref{1.8.25}, {\rm (ii)}, yields the result.

{\rm (ii)} If $A$ is a set with gaps, then there exist $x < y$ such
that the open interval $(x,y)$ is empty. Then
\[
A=(\leftarrow, x]\cup [y,\rightarrow)
\]
hence the set $A$ is not connected --- a contradiction.

\end{proof}

After passing to the dual structure $A^{op}$, part {\rm (i)} of
Proposition~\ref{1.8.30} yields

\begin{corollary}\label{1.8.35}  Let $A$ be a non-empty
linearly ordered set endowed with the interval topology. If the
topological space $A$ is connected, then every non-empty bounded
from below subset $B\subset A$ has a greatest lower bound.

\end{corollary}

Combining Proposition~\ref{1.8.16}, Proposition~\ref{1.8.30},
Corollary~\ref{1.8.35}, and taking into account Corollary~{1.8.10},
we obtain immediately

\begin{proposition}\label{1.8.36}  Let $A$ be a non-empty
linearly ordered set endowed with the interval topology. If the
topological space $A$ is connected, then every non-empty bounded and
closed subset $B\subset A$ is compact.

\end{proposition}

\begin{corollary}\label{1.8.37} Every bounded and closed interval
is compact.

\end{corollary}

\begin{proposition}\label{1.8.40} If the conditions
{\rm (i)}, {\rm (ii)} of Proposition~\ref{1.8.30} hold, then every
closed interval $[x,y]$, $x < y$, is connected.

\end{proposition}

\begin{proof} Suppose that conditions {\rm (i)}, {\rm (ii)} hold
and let $I=[x,y]$, $x < y$, be a closed interval which is a disjoint
union of two non-empty closed (with respect to the induced topology)
sets $X$ and $Y$ with $y\in Y$. Since $X$ is bounded from above by
$y$, it possesses a least upper bound $z=\sup_AX$ and we have $z\leq
y$. Moreover, there exists a $x'\in X$ and then the inequalities
$x\leq x'$, $x'\leq z$ imply $z\in I$. Since $X$ is closed in $I$,
Corollary~\ref{1.8.10} yields $z\in X$. Both possibilities $z=y$ and
$z < y$ lead to the contradiction $z\in X\cap Y$, the second one by
using the inclusion $(z,y)\subset Y$, Proposition~\ref{1.8.5}, and
the closeness of $Y$.

\end{proof}

Corollary~\ref{1.8.37} and Proposition~\ref{1.8.40} imply

\begin{corollary}\label{1.8.41} Let $A$ be a non-empty
linearly ordered set endowed with the interval topology. If the
topological space $A$ is connected, then $A$ is locally compact and
locally connected.

\end{corollary}

\begin{proposition}\label{1.8.42} If $A$ is a non-empty linearly
ordered set endowed with the interval topology and if the two
conditions {\rm (i)} and {\rm (ii)} from Proposition~\ref{1.8.30}
hold, then $A$ is connected.

\end{proposition}

\begin{proof} Proposition~\ref{1.8.40} and Proposition~\ref{1.8.20}
imply that the set $A$ is connected.

\end{proof}

Propositions~\ref{1.8.30} and~\ref{1.8.42} yield immediately

\begin{corollary}\label{1.8.43} let $A$ be a non-empty linearly
ordered set endowed with the interval topology. The topological
space $A$ is connected if and only if the two conditions {\rm (i)}
and {\rm (ii)} from Proposition~\ref{1.8.30} hold.

\end{corollary}

\begin{proposition}\label{1.8.45} Let $A$ be a non-empty linearly
ordered set endowed with the interval topology. If $A$ is connected,
then a subset of $A$ is connected if and only if it is an interval
(bounded or not bounded).

\end{proposition}

\begin{proof} Proposition~\ref{1.8.30} and Proposition~\ref{1.8.40}
yield that every bounded closed interval is connected. Now, let $I$
be any interval and let $x, y\in I$, $x < y$. Then $[x,y]\subset I$
and the closed interval $[x,y]_I=[x,y]$ is a connected subset of
$I$. According to Proposition~\ref{1.8.20}, the interval $I$ is a
connected subset of $A$. Now, let $B\subset A$ be a connected subset
of $A$. If $B=\emptyset$ or $B=\{y\}$, then $B=(x,x)$, for any $x\in
A$, or, $B=[y,y]$, respectively. Now, let $x,y\in B$, $x < y$ and
let $z\in A$, $x < z < y$. If $z\notin B$, then
$J=(\leftarrow,z]_B=(\leftarrow,z)_B$ is a non-empty open and closed
subset of $B$. The connectedness of $B$ implies $J=B$ and this
equality contradicts $y\notin J$. Thus, $z\in B$ and for any $x,y\in
B$, $x < y$, we have $[x,y]_B=[x,y]$. If $B$ is not bounded from
below and from above we obtain $B=A=(\leftarrow,\rightarrow)$. If
the subset $B$ is bounded from above, not bounded from below, and if
$z=\sup B$, we obtain $B=(\leftarrow,z)$ or $B=(\leftarrow,z]$ in
case $z\notin B$ or $z\in B$, respectively. Dually, if $B$ is
bounded from below, not bounded from above and if $z=\inf B$, we
obtain $B=(z,\rightarrow)$ or $B=[z,\rightarrow)$ in case $z\notin
B$ or $z\in B$, respectively. Finally, if $B$ is bounded, $x=\inf
B$, and $y=\sup B$, then $x < y$ ($B$ has at least two points), and
$B$ coincides with one of the four intervals with end-points $x$ and
$y$.

\end{proof}

\begin{lemma}\label{1.8.50} Let $A$ and $B$ be linearly
ordered sets endowed with the interval topology, let $A$ be
connected, and let $f\colon A\to B$ be a continuous map with image
$C=f(A)$.

{\rm (i)} For any $x < y$ in $A$ the image $J$ of the closed
interval $[x,y]$ is an interval in $C$, which contains the closed
interval in $C$ with endpoints $f(x)$ and $f(y)$.

If, in addition, $f$ is an injective map, then:

{\rm (ii)} The image $J$ coincides with the closed interval in $C$
with endpoints $f(x)$ and $f(y)$.

{\rm (iii)} The map $f$ is strictly monotonic on any three-element
subset of $A$.

\end{lemma}

\begin{proof} {\rm (i)} The image $C$ is a connected subset of $B$ and
Proposition~\ref{1.8.45} yields that $J$ is an interval (bounded or
not bounded) in $C$. Moreover, $f(x), f(y)\in J$, so the closed
interval in $C$ with endpoints $f(x)$ and $f(y)$ is a subset of $J$.

{\rm (ii)} Let $f$ be an injective map and let $z\in (x,y)$. First,
let us suppose that $f(x) < f(y)$. If $f(z) < f(x)$, then in accord
with part {\rm (i)}, there exists $x'\in [z,y]$ such that
$f(x)=f(x')$ and the injectivity of $f$ contradicts the inequality
$x < x'$. The case $f(y) < f(z)$ can be treated similarly. Thus, the
inequality $f(x) < f(y)$ implies $f(x) < f(z) < f(y)$. Following the
same way, the inequality $f(y) < f(x)$ implies $f(y) < f(z) < f(x)$.
In both cases we obtain that $J$ coincides with the closed interval
in $C$ with endpoints $f(x)$ and $f(y)$.

{\rm (iii)} Let $T=\{x,y,z\}$ be a three-element subset of $A$, let
$x < z < y$, and let, for example, $f(x) < f(y)$. The inequality
$f(x) < f(y) < f(z)$ (respectively, $f(z) < f(x) < f(y)$) and part
{\rm (ii)} yield the existence of an element $y'\in (x,z)$
(respectively, $x'\in (z,y)$) such that $f(y)=f(y')$ (respectively,
$f(x)=f(x')$) which produces contradiction to the injectivity of
$f$. Thus, $f(x) < f(z) < f(y)$ and $f$ is strictly increasing on
$T$. By considering the dual order on $Y$, in the case $f(y) < f(x)$
we prove that $f$ is strictly decreasing on $T$.

\end{proof}

\begin{corollary}\label{1.8.55} Let $A$ and $B$ be linearly
ordered sets endowed with the interval topology, let $A$ be a
connected topological space, and let $f\colon A\to B$ be a map with
image $C=f(A)$. Then $f$ is a homeomorphism of $A$ onto the subspace
$C$ if and only if $f$ is continuous and strictly monotonic.

\end{corollary}

\begin{proof} We suppose that the set $A$ has at least two
elements, the case of a singleton $A$ being clear. Let $f$ be a
continuous and strictly monotonic map. After eventual change of the
poset structure of $B$ with its dual $B^{op}$ (this leaves the
interval topology on $B$ invariant), we can suppose that $f$ is a
strictly increasing map. Proposition~\ref{0.7.10} implies that $f$
is an isomorphism of posets of $A$ and $C$. Therefore $f$ maps the
open interval $(x,y)$ in $A$ onto the open interval $(f(x),f(y))$ in
$C$. Thus, the bijection $f\colon A\to C$ is a continuous and open
map. In particular, its inverse $f^{-1}\colon C\to A$ is a
continuous map, too.

Now, let $f$ be a homeomorphism of $A$ onto its image $C$. In
particular, $f$ is an injective and continuous map. Let us suppose
that there exist two pairs $x,y\in A$, $x < y$, and $x',y'\in A$,
$x' < y'$, such that $f(x) < f(y)$ and $f(y') < f(x')$. In case the
intersection $\{x,y\}\cap\{x',y'\}$ is a singleton,
Lemma~\ref{1.8.50}, {\rm (iii)}, yields that $f$ is strictly
monotonic on the triple $\{x,y\}\cup\{x',y'\}$ which is a
contradiction. Otherwise, $f$ is strictly increasing on the triple
$\{x,y,x'\}$ and strictly decreasing on the triple $\{x,x',y'\}$,
which again is a contradiction because
$\{x,y,x'\}\cap\{x,x',y'\}=\{x,x'\}$.

\end{proof}

\begin{corollary}\label{1.8.70} Let $A$ be a countable linearly
ordered set endowed with the interval topology.

{\rm (i)} There exists a strictly increasing homeomorphism of $A$
onto one of the intervals of the rational line $\hbox{\ccc Q}$ with
endpoints $0$ and $1$, endowed with its interval topology.

{\rm (ii)} There exists a strictly increasing homeomorphism of $A$
onto a subspace of $\hbox{\ccc Q}$.

\end{corollary}

\begin{proof} {\rm (i)} According to Corollary~\ref{0.9.45}, there
exists a strictly increasing map $f$ of $A$ onto an interval $I$ in
$\hbox{\ccc Q}$ with endpoints $0$ and $1$, which turns out to be an
poset isomorphism of $A$ onto $I$, because of
Proposition~\ref{0.7.10}. Now, Proposition~\ref{1.8.4} yields that
$f$ is a homeomorphism of $A$ and $I$ endowed with their interval
topologies.

{\rm (ii)} In accord with Examples~\ref{1.8.400}, (3), the map $f$
is a strictly increasing homeomorphism of $A$ onto a subspace of
$\hbox{\ccc Q}$.

\end{proof}

\begin{remark}\label{1.8.75} {\rm If $A$ is a countable partially
ordered set endowed with the interval topology, then it is not
necessarily true that it can be embedded as a subspace of the
rational line $\hbox{\ccc Q}$. Something more, $A$ can not be
embedded as a subspace of any linearly ordered set $L$ endowed with
its interval topology. Indeed, let $A$ be the countable partially
ordered set with partial order $\leq_A$ from Examples~\ref{1.8.400},
(4), and suppose that there exists a linearly ordered set $L$  with
partial order $\leq_L$, and a strictly increasing homeomorphism $f$
of $A$ onto a subspace of $L$. Using $f$, we can identify $A$ and
its image and can suppose that $A\subset L$ is a subspace of $L$
such that $a<_A b$ implies $a<_L b$. In other words, the only
inequalities in $A$ are
\[
\frac{1}{i}<_A\frac{1}{j}, \hbox{\ }i,j\in\hbox{\ccc N}, \hbox{\
}1\leq j < i\leq n,
\]
\[
0<_A\frac{1}{i},\hbox{\ } i\in\hbox{\ccc N}, \hbox{\ }1\leq i\leq n,
\]
and we have
\[
\frac{1}{i}<_L\frac{1}{j}, \hbox{\ }i,j\in\hbox{\ccc N}, \hbox{\
}1\leq j < i\leq n,
\]
\[
0<_L\frac{1}{i},\hbox{\ } i\in\hbox{\ccc N}, \hbox{\ }1\leq i\leq n.
\]
Let $r_1,r_2,r_3\in A$ be three pairwise different rational numbers
strictly greater than $\frac{1}{2}$ and strictly less than $1$. We
can suppose that $r_1<_L r_2<_L r_3$. then $U=(r_1, r_3)_L\cap A$ is
a non-empty open set of $A$, $r_2\in U$, and this contradicts the
form of the open sets in $A$, established in Examples~\ref{1.8.400},
(4).

}

\end{remark}

Let $U$ be the set of the four intervals in $\hbox{\ccc R}$ with
endpoints $0$, $1$. For each $I\in U$ we set $I_{\hbox{\cccc
Q}}=I\cap \hbox{\ccc Q}$. Let $A$ be a linearly ordered set, let
$E$ be the set of its extremal elements (the least and the
greatest element of $A$, if exist), and let $C$ be a proper
countable subset of $A$, which is order dense in $A$. We can
suppose that $E\subset C$. Then the linearly ordered set $A$ has
no gaps and hence $C$ itself has no gaps. Now,
Corollary~\ref{1.8.70} yields the existence of a strictly
increasing homeomorphism $f$ of $A$ onto some interval
$I_{\hbox{\cccc Q}}$ in the rational line $\hbox{\ccc Q}$ where
$I\in U$. Moreover, $f(E)\subset\{0,1\}$. Let us set
$C'=C\backslash E$. Given $a\in A\backslash E$, we have the
decomposition
\[
C'=
  \left\{
\begin{array}{ll}
((\leftarrow,a)_A\cap C')\cup ((a,\rightarrow)_A\cap C')\cup\{a\}
&
\mbox{if $a\in C'$}\\
((\leftarrow,a)_A\cap C')\cup ((a,\rightarrow)_A\cap C') &
\mbox{if $a\in A\backslash C$}
\end{array}
\right.
\]
and denote $I(a)_-=f((\leftarrow,a)\cap C')$,
$I(a)_+=f((a,\rightarrow)\cap C')$. Then the pair
$(I(a)_-,I(a)_+)$ is a Dedekind cut. If $a\in C'$, then
$I(a)_-=(0,f(a))$, $I(a)_+=(f(a),1)$, and
\[
I^\circ=I(a)_-\cup I(a)_+\cup\{f(c)\},
\]
where $I^\circ$ is the open interval in $\hbox{\ccc Q}$ with
endpoints $0$ and $1$. Thus, for any $a\in C'$ the Dedekind cut
$(I(a)_-,I(a)_+)$ represents the rational number $f(a)\in I^\circ$
and we identify them: $f(a)=(I(a)_-,I(a)_+)$. If $a\in A\backslash
C$, then
\[
I^\circ=I(a)_-\cup I(a)_+
\]
and the Dedekind cut $(I(a)_-,I(a)_+)$ represents an irrational
number $1<\alpha<1$: $\alpha=(I(a)_-,I(a)_+)$. We set $F(a)=\alpha$
for any $a\in A\backslash C$, $F(a)=f(a)$ for any $a\in C$, and
obtain a map
\begin{equation}
F\colon A\to I\label{1.8.79}
\end{equation}
which is an extension of $f$ on $A$.

\begin{lemma}\label{1.8.800}  The map $F$ from~(\ref{1.8.79})
is strictly increasing and extends the homeomorphism $f\colon C\to
I_{\hbox{\cccc Q}}$ on $A$.
\end{lemma}

\begin{proof} It is enough to show that $F$ is a strictly
increasing map. Let $a,b\in A$ with $a < b$. Since $C$ is order
dense in $A$, there exists an element $c\in C$ with $a < c < b$.
Then $f(c)\in I(a)_+\cap I(b)_-$ and this implies $F(a)<F(b)$.

\end{proof}

\begin{theorem}\label{1.8.80}  Let $A$ be a linearly
ordered set endowed with the interval topology, let $A$ be a
connected topological space, and let $C$ be a countable subset which
is dense in $A$. Then there exists a strictly increasing
homeomorphism of $A$ onto one of the intervals $I\in U$, which maps
$C$ onto the interval $I_{\hbox{\ccc Q}}$.

\end{theorem}

\begin{proof} We can suppose $E\subset C$. In accord with
Proposition~\ref{1.8.30}, {\rm (ii)}, $A$
has no gaps and under this condition Corollary~\ref{0.9.50}
guarantees that $C$ has a bounded from above subset $B$ which has
no least upper bound in $C$. In particular,
Proposition~\ref{1.8.30} yields $C\neq A$ because $A$ is
connected. Lemma~\ref{1.8.800} implies that the map $F$
from~(\ref{1.8.79}) is strictly increasing and extends the
homeomorphism $f\colon C\to I_{\hbox{\cccc Q}}$ on $A$. Now, we
will prove the surjectivity of $F$. Let $\alpha$, $1<\alpha<1$, be
an irrational number, represented by the Dedekind cut $(I_-,I_+)$
of the open interval $I^\circ$, so
\[
I^\circ=I_-\cup I_+,
\]
and let $C_-=f^{-1}(I_-)$, $C_+=f^{-1}(I_+)$. Then $C_-$ and $C_+$
are not empty,
\begin{equation}
C'=C_-\cup C_+, \label{1.8.85}
\end{equation}
and $C_-\cap C_+=\emptyset$. Let $a\in C_-$ and $b\in C_+$. Since
$f(a)\in I_-$ and $f(b)\in I_+$, we have $f(a)<f(b)$ and this
inequality implies $a<b$. Thus, $C_-$ consists of lower bounds of
$C_+$ and $C_+$ consists of upper bounds of $C_-$. Let $a_-=\sup
C_-$ and $a_+=\inf C_+$. Then we have $a_-\leq a_+$. Since there
are no gaps in $A$ and since $C$ is dense in $A$, the assumption
$a_- < a_+$ implies existence of an element $c\in (a_-, a_+)\cap
C'$, which contradicts the equality~(\ref{1.8.85}). Therefore
$a=a_- = a_+$ and
\begin{equation}
C_-\subset (\leftarrow,a)\cap C',\hbox{\ } C_+\subset
(a,\rightarrow)\cap C'. \label{1.8.90}
\end{equation}

If $a\in C'$, then the equality~(\ref{1.8.85}) contradicts the
equality~(\ref{1.8.90}). Thus, $a\in A\backslash C$ and we have
\[
C'=C_-\cup C_+\subset (\leftarrow,a)\cap C'\cup
(a,\rightarrow)\cap C'=C'.
\]
The last inclusions imply
\[
C_-=(\leftarrow,a)\cap C',\hbox{\ } C_+=(a,\rightarrow)\cap C'
\]
and then $I(a)_-\subset I_-$, $I(a)_+\subset I_+$, hence
\[
F(a)=(I(a)_-,I(a)_+)=(I_-,I_+)=\alpha
\]
and therefore the map $F$ is surjective. Thus, we have a strictly
increasing bijection $F\colon A\to I$ which is a homeomorphism of
$A$ onto the subspace $I$ of the real line $\hbox{\ccc R}$ because
of Proposition~\ref{1.8.4}, {\rm (ii)}.

\end{proof}

\end{document}